\documentclass[12 pt, a4paper]{article}
\usepackage{amssymb, amsmath}
\usepackage{amsthm}
\usepackage[margin=1in]{geometry}
\usepackage{enumerate}
\allowdisplaybreaks[1]
\numberwithin{equation}{section}

\newtheorem{thm}{Theorem}[section]
\newtheorem{cor}{Corollary}[section]

\newtheorem{pro}{Proposition}[section]

\newtheorem{rem}{Remark}[section]
\newtheorem{lemma}{Lemma}[section]
\usepackage{graphicx}
\usepackage{mathrsfs}

\begin{document}
\markboth{R. Rajkumar and P. Devi}{Intersection graph of subgroups}
\title{\bf Toroidality and projective-planarity of intersection graphs of subgroups of finite groups}

\author{R. Rajkumar\footnote{e-mail: {\tt rrajmaths@yahoo.co.in}},\ \ \
P. Devi\footnote{e-mail: {\tt pdevigri@gmail.com}}\\
{\footnotesize Department of Mathematics, The Gandhigram Rural Institute -- Deemed University,}\\ \footnotesize{Gandhigram -- 624 302, Tamil Nadu, India.}\\[3mm]
}
\date{}

\maketitle

\begin{abstract}
Let $G$ be a group. \textit{The intersection graph of subgroups of $G$}, denoted by $\mathscr{I}(G)$, is a graph with all the
proper subgroups of $G$ as its vertices and two distinct vertices  in $\mathscr{I}(G)$ are adjacent if and only if
the corresponding subgroups having a non-trivial intersection in $G$.
In this paper, we classify the finite groups whose intersection graph of subgroups are toroidal or projective-planar.
In addition, we classify the finite groups whose intersection graph of subgroups are one of bipartite, complete bipartite, tree, star graph,  unicyclic, acyclic,  cycle, path or totally disconnected. Also we classify the finite groups whose intersection graph of subgroups does not contain one of  $K_5$, $K_4$, $C_5$, $C_4$, $P_4$, $P_3$,  $P_2$, $K_{1,3}$, $K_{2,3}$ or $K_{1,4}$ as a subgraph.
We estimate the  girth of the intersection graph of subgroups of finite groups. Moreover, we characterize some finite groups by using their intersection graphs. Finally, we obtain the clique cover number of the intersection graph of subgroups of groups and show that intersection graph  of subgroups of groups are weakly $\alpha$-perfect.

\end{abstract}

\section{Introduction} \label{sec:int}
Let $\mathcal{F}=\{S_i~|~i\in I\}$ be an arbitrary family of sets. The intersection graph of $\mathcal{F}$ is a graph having the elements of $\mathcal{F}$ an
its vertices and two vertices $S_i$ and $S_j$ are adjacent if and only if $i\neq j$ and $S_i\cap S_j\neq \{\emptyset\}$. For the properties of these graphs and some
special class of intersection graphs, we refer the reader to \cite{mckee}. In the past fifty years,
it has been an interesting topic for mathematicians, when the members of $\mathcal{F}$ have some specific algebraic
structures. Especially, they investigate on the interplay between the algebraic properties of algebraic structures, and the graph theoretic properties of their
intersection graphs. In this direction, in 1964 Bosak~\cite{bosak} initiated the study of the intersection graphs of semigroups. Later,
Cs$\acute{a}$k$\acute{a}$ny and Poll$\acute{a}$k ~\cite{csak} defined the intersection graph of subgroups of a finite group.

Let $G$ be a group.  \textit{The intersection graph of subgroups of $G$}, denoted by $\mathscr{I}(G)$, is a graph with all the
proper subgroups of $G$ as its vertices and two distinct vertices  in $\mathscr{I}(G)$ are adjacent if and only if
the corresponding subgroups have a non-trivial intersection in $G$.

In~\cite{zelinka} Zelinka
made some investigations on the intersection graphs of subgroups of finite abelian groups.  Motivated by these, many authors have defined, and studied the intersecting graphs on several algebraic structures, viz., rings, vector spaces, modules, and contributed interesting results. See, for instance \cite{ akbari_2,  akbari_1,  akbari, Chakara, laison, Shen, yaraneri} and the references therein.

Embeddability of graphs, associated with algebraic structures, on topological surfaces is considered in several recent papers \cite{hung,   zoran, zoran1, raj2015}. Planarity of intersection graphs of subgroups finite groups were studied by Sel\c{c}¸uk Kayacan \emph{et al.} in \cite{sel},  and by H. Ahmedi \emph{et al.} in \cite{hadi}. Planarity of intersection graphs of ideals of rings, and submodules of modules were studied in \cite{jafari1, yaraneri}.

  A natural question arise in this direction is the following: Which groups have their intersection graph of subgroups is of genus one, that is toroidal, or of nonorientable genus one, that is projective-planar ?. In this paper, we answer this question in the case of finite groups by classifying the finite groups whose intersection graph of subgroups is toroidal or projective-planar (see Theorem~\ref{intersecting graph t12} in Section~\ref{sec: 5} below).  As a consequence of this research, we also classify finite groups whose intersection graph of subgroups is in some class of graphs (see Theorem ~\ref{intersecting graph t233} and Corollary~\ref{intersecting graph c1} in Section~\ref{sec: 5} below) and characterize some finite groups by using their intersection graphs (see Corollary~\ref{intersecting graph c2} in Section~\ref{sec: 5} below), which are some of the main applications of these results for the group theory. Also we estimate the grith of the intersection graph of finite groups. Finally, we obtain the clique covering number of the intersection graph of subgroups of groups and show that intersection graph  of subgroups of groups are weakly $\alpha$-perfect.

\section{Preliminaries and notations}\label{sec:pre}
In this section, we first recall some notation, and results in graph theory, which are used later in the subsequent sections. We use standard basic graph theory terminology and notation (e.g., see \cite{arthur}).
Let $G$ be a simple graph with a vertex set $V$ and an edge set $E$.
 $G$ is said to be \textit{bipartite} if $V$ can be partitioned into two subsets $V_1$ and $V_2$ such that every edge of $G$ joins a vertex of $V_1$ to a vertex of $V_2$. Then $(V_1, V_2)$ is called a \textit{bipartition} of $G$.
Moreover, if every vertex of $V_1$ is adjacent to every vertex of $V_2$, then $G$ is called \textit{complete bipartite} and is denoted by $K_{m,n}$, where $|V_1|=m$, $|V_2|=n$. In particular, $K_{1,n}$ is a \emph{star graph}. $G$ is said to be \emph{complete} if each pair of distinct vertices in $G$ are adjacent.  The complete graph on $n$ vertices is denoted by $K_n$. $G$ is said to be \textit{totally disconnected} if its edge set is empty.  A \textit{path} connecting two vertices $u$ and $v$ in $G$ is a finite sequence  $(u=) v_0, v_1, \ldots, v_n (=v)$ of distinct vertices (except,
possibly, $u$ and $v$) such that $u_i$ is adjacent to $u_{i+1}$ for all $i=0, 1, \ldots , n-1$.
A path is a \emph{cycle} if $u=v$. The length of a path or a cycle is the number of edges in it. A path or a cycle of length $n$ is denoted by $P_n$ or $C_n$, respectively. A graph is \emph{unicyclic} if it has exactly one cycle, and is \emph{acyclic} if it has no cycles. The \emph{girth} of a graph $G$, denoted by $girth(G)$, is the length of the shortest cycle in $G$, if it exist; otherwise $girth(G) = \infty$.

We define a graph $G$ to be $X$-\textit{free} if $G$ has no subgraph isomorphic to a given graph $X$. An \emph{independent set} of a graph $G$ is a subset of the vertices of $G$ such that no two vertices in the subset are adjacent. The \emph{independence
number} of $G$, denoted by $\alpha(G)$, is the cardinality of a maximum independent set of $G$. A \emph{clique} of a graph $G$ is a complete subgraph of $G$. The \emph{clique cover number} of $G$, denoted by $\theta(G)$, is the minimum number of cliques in $G$ which cover all the vertices of $G$ (not necessarily all the edges of $G$). $G$ said to be \emph{weakly $\alpha$-perfect} if $\alpha(G)=\theta(G)$. For two graphs $G$ and $H$, $G \cup H$ denotes  disjoint union of $G$ and $H$, $G+H$ denotes a graph with the vertex set consist of vertices of $G$ and $H$ and edge set having all the lines
joining vertices of $G$ to vertices of $H$. $\overline{G}$ denotes the complement of a graph $G$, and, for an integer $n\ge 1$, $nG$ denotes the graph having $q$ disjoint copies of $G$.

Let $\mathbb{S}_n$ denote the surface obtained from the sphere by attaching $n$ handles. A graph is said to be \textit{embeddable} on a topological surface if it can be drawn on the surface such that no two edges cross.
The (orientable) \textit{genus} of a graph $G$, denoted by $\gamma(G)$, is the smallest non-negative integer $n$ such that $G$ can be embedded on $\mathbb{S}_n$.
$G$ is \textit{planar} if $\gamma(G)=0$ and \textit{toroidal} if $\gamma(G)=1$.
A crosscap is a circle (on the surface) such that all its pairs of opposite points are identified, and
the interior of this circuit is removed. Let $\mathbb{N}_k$ denote the sphere with $k$ added crosscaps.
For non-orientable topological surfaces (e.g., the projective plane, Klein bottle, etc.),
the \textit{nonorientable genus} of $G$ is the smallest integer $k$ such that $G$ can be embedded on $\mathbb{N}_k$, and it is denoted by $\overline{\gamma}(G)$.
$\mathbb{N}_1$ is the projective plane. Respectively, a graph $G$ is \textit{projective-planar} if $\overline{\gamma}(G) =1$.

A \textit{topological obstruction} for a surface is a graph $G$ of minimum vertex degree at least three such that $G$ does not embed on the surface,
but $G-e$ is embeddable on the surface for every edge $e$ of $G$. A \emph{minor-order obstruction} $G$ is a topological obstruction with the additional
property that, for each edge $e$ of $G$, $G$ with the edge $e$ contracted embeds on the surface.

The following results are used in the subsequent sections.

\begin{thm}\label{genus 110}(\cite[Theorems 6.37, 6.38 and 11.19, 11.23]{arthur})
\
\begin{enumerate}[{\normalfont (1)}]
\item $\gamma(K_n)=\displaystyle\left\lceil \frac{(n-3)(n-4)}{12}\right\rceil$, $n\geq 3$;

    $\gamma(K_{m,n})=\displaystyle\left\lceil \frac{(m-2)(n-2)}{4}\right\rceil$, $m$, $n\geq 2$.
\item $\overline{\gamma}(K_n)=\left\{
   	\begin{array}{ll}
   		\displaystyle\left\lceil \frac{(n-3)(n-4)}{6} \right\rceil, & \mbox{~~if~ } n\geq 3, n\neq 7;  \\
   		3, & \mbox{~~if~ } n=7;
   	\end{array}
   \right.$

$\overline{\gamma}(K_{m,n})=\displaystyle\left\lceil \frac{(m-2)(n-2)}{2}\right\rceil$, $m$, $n\geq 2$.
\end{enumerate}
\end{thm}
By Theorem~\ref{genus 110}, one can see that $\gamma(K_n)>1$ for $n\geq 8$, $\overline{\gamma}(K_n)>1$ for $n\geq 7$, $\gamma(K_{m,n})>1$  if either $m\geq4$, $n\geq 5$ or $m\geq 3$, $n\geq 7$, and $\overline{\gamma}(K_{m,n})>1$ if either $m\geq3$, $n\geq5$ or $m=n=4$. Neufeld and Myrvold \cite{pra} have shown the following.

\begin{thm}(\cite[Figure 4, p. 578]{pra})\label{genus 501}
There are exactly three eight-vertex obstructions $\mathcal{A}_1, \mathcal{A}_2, \mathcal{A}_3$ for the torus, each of them being topological and minor-order (see Figure~\ref{fig:f2}).
\end{thm}

\begin{figure}[ ht ]
 \begin{center}
 \includegraphics[scale=.6]{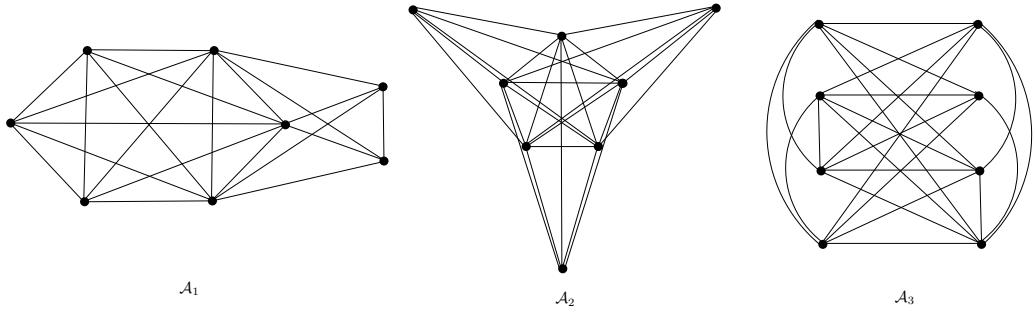}
 \caption{The eight vertex obstructions for the torus.}
\label{fig:f2}
 \end{center}
 \end{figure}

Note that the graph $\mathcal {A}_1$ in Figure~\ref{fig:f2} has $K_{3,5}$ as a subgraph, so it follows that if a graph $G$ has $\mathcal {A}_1$ as a subgraph, then $\gamma(G)>1$ and $\overline{\gamma}(G)>1$. In this paper, we repeatedly use this fact.

Gagarin \emph{et al.}~\cite{and} have found all the toroidal obstructions for the graphs containing no subdivisions of $K_{3,3}$ as a subgraph. These graphs coincide with the graphs containing no $K_{3,3}$-minors and are called \emph{with no $K_{3,3}$'s}.

\begin{thm}(\cite[Theorem~3, p. 3628]{and}) \label{genus 500}
There are exactly four minor-order obstructions with no $K_{3,3}$'s for the torus, precisely,
$\mathcal{B}_1,\mathcal{B}_2,\mathcal{B}_3,\mathcal{B}_4$ shown in Figure~\ref{fig:int f3}
as a minor.
\end{thm}

\begin{figure}[ ht ]
\begin{center}
\includegraphics[scale=.7]{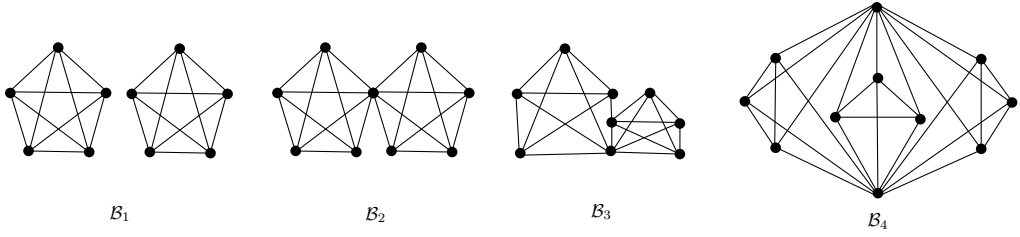}
\caption{The minor-order obstructions with no $K_{3,3}$'s for the torus.}
\label{fig:int f3}
\end{center}
\end{figure}
Notice that all the obstructions in Theorems~\ref{genus 501} and \ref{genus 500} are obstructions for toroidal graphs in general, which are very numerous (e.g., see \cite{and}).

In \cite{Chakara}, Chakarabarthy \emph{et al.}  proved the following results on the intersection graph of ideals of the ring $\mathbb Z_n$.

\begin{thm}(\cite{Chakara})\label{l1}
Let $p$, $q$, $r$ be distinct primes. Then
\begin{enumerate}[{\normalfont (1)}]
\item $\mathscr{I}(\mathbb Z_n)$ is planar if and only if $n$ is one of $p^\alpha (\alpha=2,3,4,5)$, $p^\alpha q (\alpha=1, 2)$, or $pqr$.
\item $\mathscr{I}(\mathbb Z_n)$ is $K_5$-free if and only if it is planar.
\item $\mathscr{I}(\mathbb Z_n)$ is bipartite if and only if $n$ is either $p^3$ or $pq$.
\item $\mathscr{I}(\mathbb Z_n)$ is $C_3$-free if and only if $n$ is either $p^\alpha$ $(\alpha=2,3)$ or $pq$.
\end{enumerate}
\end{thm}

In \cite{sel} Sel\c{c}¸uk Kayacan \emph{et al.} have classified all the finite groups whose intersection graphs of subgroups are planar.
\begin{thm}(\cite{sel})\label{1000}
Let $G$ be a finite group and $p$, $q$, $r$ be distinct primes. Then $\mathscr{I}(G)$ is planar if and only if $G$ is one of the following groups:
\begin{enumerate}[\normalfont (1)]
\item $\mathbb Z_{p^\alpha} (\alpha=2,3,4,5)$, $\mathbb Z_{p^\alpha q} (\alpha=1, 2)$, $\mathbb Z_{pqr}$, $\mathbb Z_p \times \mathbb Z_p$, $\mathbb Z_4 \times \mathbb Z_2$, $\mathbb Z_6\times \mathbb Z_2$;
\item $Q_8$, $M_8$, $\mathbb Z_q \rtimes \mathbb Z_p$, $\mathbb Z_q \rtimes_2 \mathbb Z_{p^2}$, $A_4$;
\item $\mathcal{G}_1 \cong \langle a, b, c~|~ a^p=b^p=c^q=1, ab=ba, cac^{-1}=b,
cbc^{-1}= a^1b^{l} \rangle$, where $\bigl(\begin{smallmatrix}
  0 & -1\\ 1 & l
\end{smallmatrix} \bigr)$ has order $q$ in $GL_2(p)$, where $q~|~(p+1)$;
\item $\mathcal{G}_2 \cong\langle a,b,c~|~a^p=b^p=c^{q^2}=1, ab=ba, cac^{-1}=b^{-1}, cbc^{-1}=a^1b^l\rangle$,
 where $\bigl(\begin{smallmatrix}
  0 & -1\\ 1 & l
\end{smallmatrix} \bigr)$ has order $q^2$ in $GL_2(p)$, where $q^2~|~(p+1)$;
\item $\mathcal{G}_3 \cong\langle a,b,c~|~a^p=b^q=c^r=1, b^{-1}ab=a^\mu, c^{-1}ac=a^v, bc=cb\rangle$, $q$, $r$ are divisors of $(p-1)$, $v$, $\mu\neq 1$.
\end{enumerate}
\end{thm}

The intersection graph of subgroups of the groups listed in Theorem~\ref{1000}, are given below for use in the subsequent sections.
\begin{equation}\label{intersection graph e1}
\mathscr{I}(\mathbb Z_{p^\alpha})\cong K_{\alpha-1}, \alpha=2, 3, 4, 5.
\end{equation}
\begin{equation}\label{intersection graph eqn2}
\mathscr{I}(\mathbb Z_{pq})\cong \overline{K}_2.
\end{equation}
\begin{equation}\label{intersection graph e2}
\mathscr{I}(\mathbb Z_{p^2q})\cong K_1+(K_2\cup K_1).
\end{equation}
\begin{equation}\label{e4}
\mathscr{I}(\mathbb Z_p\times \mathbb Z_p) \cong \overline{K}_{p+1}.
\end{equation}
\begin{equation}\label{5}
\mathscr{I}(\mathbb Z_4\times \mathbb Z_2) \cong K_1 +(K_4 \cup \overline{K}_2).
\end{equation}
\begin{equation}\label{23}
	\mathscr{I}(Q_8)\cong K_4.
\end{equation}
\begin{equation}\label{e8}
\mathscr{I}(\mathbb Z_q \rtimes \mathbb Z_p) \cong \overline{K}_{q+1}
\end{equation}
 \begin{equation}\label{e9}
 \mathscr{I}(\mathbb Z_q \rtimes_2 \mathbb Z_{p^2}) \cong K_1+ (K_1 \cup q K_2),
 \end{equation}
 \begin{equation}\label{e11}
 \Gamma(\mathcal{G}_1)\cong K_{1,p+1}\cup \overline{K}_{p^2}.
 \end{equation}
 \begin{equation}\label{e10}
 \mathscr{I}(A_4) \cong K_{1,3} \cup \overline{K}_4.
 \end{equation}
 \begin{equation}\label{33}
 \mathscr{I}(\mathcal{G}_2)\cong K_1+(K_{1,p+1}\cup p^2K_2).
 \end{equation}
 \begin{figure}[ ht ]
 \begin{center}
  \includegraphics[scale=.8]{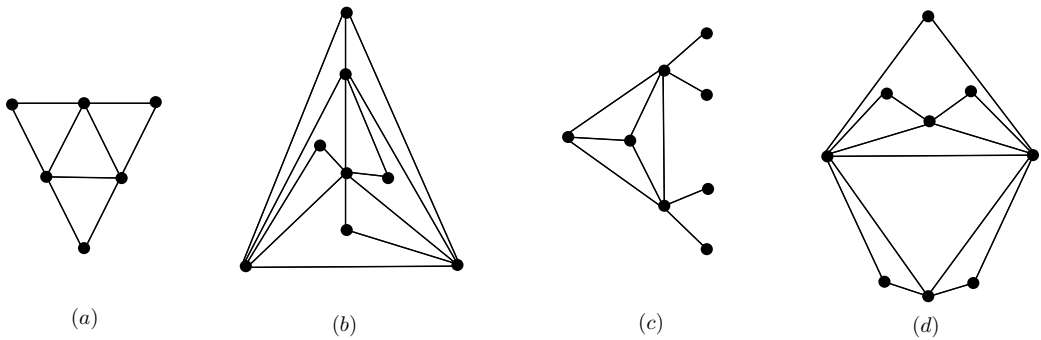}
  \caption{(a) $\mathscr{I}(\mathbb Z_{pqr})$, (b) $\mathscr{I}(\mathbb Z_6\times \mathbb Z_2)$, (c) $\mathscr{I}(M_8)$, (d) $\mathscr{I}(\mathcal{G}_3)$.}
  \label{int f1}
  \end{center}
  \end{figure}
We summarize some of the results of group theory which we use in forthcoming sections.
\begin{thm}\label{501}
\begin{enumerate}[\normalfont (i)]
\item (\cite[Theorem IV, p.129]{burn}) If $G$ is a $p$-group of order $p^n$, then the number of subgroups of order $p^s$, $1\leq s\leq n$ is congruent to 1 (mod p).
\item (\cite[Proposition 1.3]{scott}) If $G$ is a $p$-group of order $p^n$ and it has a unique subgroup of order $p^s$, $1<s\leq n$, then $G$ is cyclic or  $s=1$ and $p=2$, $G\cong Q_{2^\alpha}$.
\item (\cite[Corollary IV, p. 53]{burn}) If $G$ has a normal subgroup $H$ of order $mn$, where $m$ and $n$ are relatively prime and $N$ is a normal subgroup of $H$ of order $n$, then $N$ is also normal in $G$.
\end{enumerate}
\end{thm}

\section{Finite abelian groups}\label{sec:3}
In this section, first we classify the finite abelian groups whose intersection graphs of subgroups are toroidal or projective-planar, and next we classify the finite non-cyclic abelian groups whose intersection graphs of subgroups are one of $K_5$ free, $C_3$-free, acyclic or bipartite.
\begin{pro}\label{intersecting graph t100}
Let $G$ be a finite cyclic group and $p$, $q$ be distinct primes. Then
\begin{enumerate}
\item $\mathscr{I}(G)$ is toroidal if and only if
$G$ is one of $\mathbb Z_{p^\alpha}(\alpha=6, 7, 8)$, $\mathbb Z_{p^\alpha q}(\alpha=3, 4)$ or $\mathbb Z_{p^2q^2}$;
\item $\mathscr{I}(G)$ is projective-planar if and only if $G$ is one of $\mathbb Z_{p^\alpha}$$(\alpha=6, 7)$ or $\mathbb Z_{p^3q}$.
\end{enumerate}
\end{pro}
\begin{proof}
Let $|G|=p_1^{\alpha_ 1}p_2^{\alpha_ 2}\dots p_k^{\alpha_ k}$, where $p_i$'s are distinct primes, and $\alpha_i \geq 1$ are integers.
Note that to prove this result, it is enough to investigate the finite cyclic groups other than those listed in Theorem~\ref{l1}(1). We have to deal with the following cases.

\noindent \textbf{Case 1:} If $k=1$, then the number of proper subgroups of $G$ is $\alpha_1-1$ and so
\begin{equation}\label{200}
\mathscr{I}(G)\cong K_{\alpha_1-1}.
\end{equation}

\noindent It follows that
$\gamma(\mathscr{I}(G))=1$ if and only if $\alpha_1=6$, 7, 8. and  $\overline{\gamma}(\mathscr{I}(G))=1$ if and only if $\alpha_1=6$, 7.

\noindent \textbf{Case 2:} If $k=2$ and $\alpha_2= 1$, then we consider the following subcases:

\noindent \textbf{Sub case 2a:} $\alpha_1=3$. Let $H_i$, $i=1$, $\ldots$, 6 be the subgroups of $G$ of order $p_1$, $p_1^2$, $p_1^3$, $p_2$, $p_1p_2$,
$p_1^2p_2$ respectively. Here $H_1$ is a subgroup of $H_2$, $H_3$, $H_5$, $H_6$; $H_4$ is a subgroup of $H_5$, $H_6$; $H_4$ has trivial intersection with $H_i$,
$i=1$, 2, 3. Therefore,
\begin{equation}\label{201}
\mathscr{I}(G)\cong K_2+(K_3\cup K_1),
\end{equation} which is a subgraph of $K_6$ and it contains $K_5$. So $\gamma(\mathscr{I}(G))=1$ and $\overline{\gamma}(\mathscr{I}(G))=1$.

\noindent\textbf{Sub case 2b:} $\alpha_1=4$. Let $H_i$, $i=1$, $\ldots$, 8 be the subgroups of $G$ of order
$p_1$, $p_1^2$, $p_1^3$, $p_1^4$, $p_2$, $p_1p_2$,
$p_1^2p_2$, $p_1^3p_2$ respectively. Here $H_1$ is a subgroup of $H_2$, $H_3$, $H_4$, $H_6$, $H_7$, $H_8$; $H_5$ is a subgroup of $H_i$, $i=6$, 7, 8;
$H_5$ has trivial intersection with $H_i$, $i=1$, 2, 3, 4.
It follows that
\begin{equation}\label{2000}
\mathscr{I}(G)\cong K_3+(K_4\cup K_1).
\end{equation}
Also $\gamma(\mathscr{I}(G))=1$ and the corresponding toroidal embedding is shown in Figure~\ref{fig:int f4}. Since
$\mathscr{I}(G)$ contains $K_7$ as a subgraph, it follows that $\overline{\gamma}(\mathscr{I}(G))>1$.

\begin{figure}[ht]
\begin{center}
\includegraphics[scale=.8]{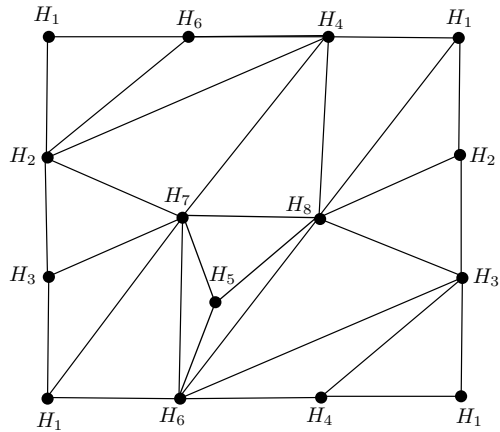}
\caption{An embedding of $\mathscr{I}(\mathbb Z_{p_1^4p_2})$ in torus.}
\label{fig:int f4}
\end{center}
\end{figure}

\noindent\textbf{Sub case 2c:} $\alpha_1\geq5$. Let $H_i$, $i=1$, $\ldots$, 8 be the subgroups of $G$ of order
$p_1$, $p_1^2$, $p_1^3$, $p_1^4$, $p_1^\alpha$, $p_1p_2$,
$p_1^2p_2$, $p_1^3p_2$ respectively. Since $H_1$ is a subgroup of $H_i$, $i=2$, $\ldots$, 8, it follows that
they form $K_8$ as a subgraph of $\mathscr{I}(G)$ and so $\gamma(\mathscr{I}(G))>1$,
$\overline{\gamma}(\mathscr{I}(G))>1$.

\noindent \textbf{Case 3:} $k=2$ and $\alpha _2\geq 2$. Now we have to deal with the following sub cases:

\noindent\textbf{Sub case 3a:} $\alpha_1=\alpha_2=2$. Let $H_i$, $i=1$, $\ldots$, 7 be the subgroups of $G$ of order
$p_1$, $p_1^2$, $p_2$, $p_2^2$, $p_1p_2$, $p_1^2p_2$,
$p_1p_2^2$ respectively. Here $H_1$ is a subgroup of $H_2$, $H_5$, $H_6$, $H_7$; $H_3$ is a subgroup of $H_4$, $H_5$, $H_6$, $H_7$; $H_1$, $H_2$ have trivial
intersection with $H_3$, $H_4$. It follows that
\begin{equation}\label{202}
\mathscr{I}(G)\cong K_3+2K_2.
\end{equation}
Moreover $\gamma(\mathscr{I}(G))=1$, since $\mathscr{I}(G)$ is a subgraph of $K_7$ and it contains $K_5$.

Also $\overline{\gamma}(\mathscr{I}(G))>1$, since $\mathscr{I}(G)$ has a subgraph isomorphic to the graph shown in Figure~\ref{fig:int f5}, which is one of the obstruction for projective-plane (e.g., see Theorem 0.1 and graph
$B_1$ of case (3.6) on p. 340 in \cite{glov}).
\begin{figure}[ ht ]
\begin{center}
\includegraphics[scale=.8]{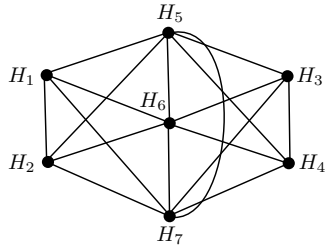}
\caption{An obstruction for the projective-plane.}
\label{fig:int f5}
\end{center}
\end{figure}

\noindent\textbf{Sub case 3b:} Either $\alpha>2$ or $\beta>2$. Let
$H_i$, $i=1$, $\ldots$, 8 be subgroups of $G$ of order
$p_1$, $p_1^2$, $p_1^3$, $p_2$, $p_2^2$, $p_1p_2$,
$p_1^2p_2$, $p_1^3p_2$ respectively. Here $H_1$ is a subgroup of $H_2$, $H_3$, $H_6$, $H_7$, $H_8$; $H_4$ is a subgroup of $H_5$, $H_6$, $H_7$, $H_8$.
It follows that $\mathscr{I}(G)$ has a subgraph $\mathcal A_1$ as shown in Figure~\ref{fig:f2}, so $\gamma(\mathscr{I}(G))>1$ and $\overline{\gamma}(\mathscr{I}(G))>1$.

Also $\mathscr{I}(G)$ contains $K_{3,5}$ as a subgraph with bipartition $X:=\{H_3$, $H_6$ $H_7\}$ and $Y:=\{H_1$, $H_2$, $H_4$, $H_5$, $H_8\}$ and so
$\overline{\gamma}(\mathscr{I}(G))>1$.

\noindent \textbf{Case 4:} If $k=3$, then we need to consider only for $\alpha_1\geq2$, $\alpha_2$, $\alpha_3\geq1$. Let $H_i$, $i=1$, $\ldots$, 8
be subgroups of $G$ of order
$p_1$, $p_1^2$, $p_1p_2p_3$, $p_1p_2$, $p_1p_3$,
$p_2p_3$, $p_1^2p_2$, $p_1^2p_3$ respectively. Here $H_1$ is a subgroup of $H_2$, $H_3$, $H_4$, $H_5$, $H_7$, $H_8$; $H_i$, $i=3$, $\ldots$, 8
intersect non-trivially with each other. It follows that $\mathscr{I}(G)$ has a subgraph $\mathcal A_1$ as shown in Figure~\ref{fig:f2}, so $\gamma(\mathscr{I}(G))>1$
 and $\overline{\gamma}(\mathscr{I}(G))>1$.

\noindent \textbf{Case 5:} If $k\geq 4$, then let $H_i$, $i=1$, $\ldots$, 8 be subgroups of $G$ of orders
$p_1$, $p_1p_2$, $p_1p_2p_3$, $p_1p_2p_4$, $p_1p_3p_4$,
$p_2p_3p_4$, $p_1p_3$, $p_1p_4$ respectively. Here $H_1$ is a subgroup of $H_2$, $H_3$, $H_4$, $H_6$, $H_7$, $H_8$; $H_5$, $H_3$, $H_4$, $H_6$, $H_7$, $H_8$
intersect non-trivially with each other. It follows that $\mathscr{I}(G)$ has a subgraph $\mathcal A_1$ as shown in Figure~\ref{fig:f2}, so $\gamma(\mathscr{I}(G))>1$ and  $\overline{\gamma}(\mathscr{I}(G))>1$

Combining  all the cases together the proof follows.
\end{proof}

\begin{pro}\label{intersecting graph t1}
 Let $G$ be a finite non-cyclic abelian group and $p$, $q$ be distinct primes. Then
\begin{enumerate}[{\normalfont (1)}]

\item $\mathscr{I}(G)$ is $K_5$-free if and only if $G$ is one of $\mathbb Z_p \times \mathbb Z_p$, $\mathbb Z_4 \times \mathbb Z_2$, or $\mathbb Z_6\times \mathbb Z_2$;


\item The following are equivalent:
\begin{enumerate}
\item $G\cong\mathbb Z_p\times \mathbb Z_p$;
\item $\mathscr{I}(G)$ is $C_3$-free;
\item $\mathscr{I}(G)$ is acyclic;
\item $\mathscr{I}(G)$ is bipartite.
\end{enumerate}
\item $\mathscr{I}(G)$ is toroidal if and only if $G$ is one of $\mathbb Z_{p^2}\times \mathbb Z_p$$(p=3,5)$ or $\mathbb Z_{3q}\times \mathbb Z_3$;
\item $\mathscr{I}(G)$ is projective-planar if and only if $G$ is either $\mathbb Z_9\times \mathbb Z_3$ or $\mathbb Z_{3q}\times \mathbb Z_3$.
\end{enumerate}
\end{pro}
\begin{proof}
We split the proof in to several cases:

\noindent\textbf{Case 1}: $G \cong \mathbb Z_p \times \mathbb Z_p$. Then by Theorem~\ref{1000}, $\mathscr{I}(G)$ is planar and by \eqref{e4}, it is acyclic.

\noindent\textbf{Case 2:} $G\cong{\mathbb Z_{p^2}} \times{\mathbb Z_p}$. Here $\langle(1,0)\rangle$, $\langle(1,1)\rangle$, $\ldots$, $\langle(1,p-1)\rangle$,
$\langle(p,0), (0,1)\rangle$, $\langle(p,0)\rangle$, $\langle(p,1)\rangle$, $\ldots$, $\langle(p,p-1)\rangle$, $\langle(0,1)\rangle$ are the only
proper subgroups of $G$. Note that $\langle(p,0)\rangle$ is a subgroup of $\langle(1,0)\rangle$, $\langle(1,1)\rangle$, $\ldots$, $\langle(1,p-1)\rangle$,
$\langle(p,0), (0,1)\rangle$; $\langle(p,0)\rangle$, $\langle(p,1)\rangle$, $\ldots$, $\langle(p,p-1)\rangle$, $\langle(0,1)\rangle$ are
subgroups of $\langle(p,0), (0,1)\rangle$; no two remaining subgroups intersect non-trivially. It follows that
\begin{equation}\label{e5}
\mathscr{I}(G) \cong K_1 +(K_{p+1} \cup \overline{K}_p).
\end{equation}

So $\mathscr{I}(G)$ contains $C_3$ as a subgraph. Note that $\mathscr{I}(G)$ is a graph obtained by attaching $p$ pendent edges to any one of the vertices of $K_{p+2}$. So  $\gamma(\mathscr{I}(G))=1$ if and only if $p=3,5$; $\overline{\gamma}(\mathscr{I}(G))=1$ if and only if $p=3$; $\mathscr{I}(G)$ is $K_5$-free if and only if $p=2$.

\noindent\textbf{Case 3:} $G\cong \mathbb Z_{pq}\times \mathbb Z_p$. If $p=2$, then
by Theorem~\ref{1000}, $\mathscr{I}(G)$ is planar and by Figure~3(b), it contains $C_3$.
If $p=3$, then $H:=\mathbb Z_p\times \mathbb Z_p$ is a subgroup of $G$. $H$ has four proper subgroups, let them be $H_1$, $H_2$, $H_3$, $H_4$. Now $H$ and its proper subgroups,  $H_5:=\mathbb Z_q\times \{e\}$, $H_iH_5$, $i=1,2,3,4$ are the only proper subgroups of $G$. Here $H_5$ is a subgroup of $H_iH_5$,
$i=1,2,3,4$; $H_iH_5$, $i=1,2,3,4$ has non-trivial intersection with $H$; $H_i=1,2,3,4$ are subgroups of $H$; no two remaining subgroups intersect nontrivially.
A toroidal embedding of $\mathscr{I}(G)$ is shown in Figure~\ref{fig:int f6},
and an embedding of $\mathscr{I}(G)$ in the projective-plane is shown in Figure~\ref{fig:int f7}. Note that $\mathscr{I}(G)$ contains $K_5$.

\begin{figure}[ ht ]
\begin{center}
\includegraphics[scale=.8]{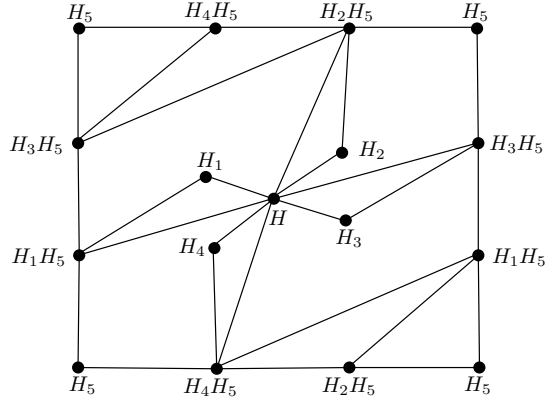}
\caption{An embedding of $\mathscr{I}(\mathbb Z_{3q}\times \mathbb Z_3)$ in the torus.}
\label{fig:int f6}
\end{center}
\end{figure}

\begin{figure}[ ht ]
\begin{center}
\includegraphics[scale=.8]{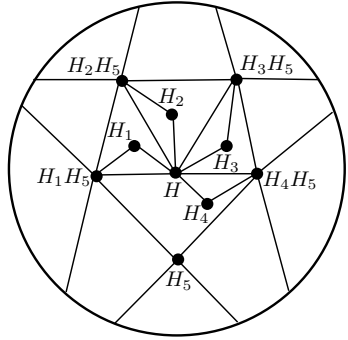}
\caption{An embedding of $\mathscr{I}(\mathbb Z_{3q}\times \mathbb Z_3)$ in projective plane.}
\label{fig:int f7}
\end{center}
\end{figure}

If $p=5$, then $H:=\mathbb Z_p\times \mathbb Z_p$ is a subgroup of $G$. $H$ has six proper subgroups, let them be $H_i$, $i=1$, $\ldots$, $6$. Now $H$ and its proper subgroups, $H_7:=\mathbb Z_q\times \{e\}$, $B_i=H_iH_7$, $i=1$, $\ldots$, 6 are the only proper subgroups of $G$. Here $H_7$ is a subgroup of $B_i$,
$i=1$, $\ldots$, 6; $A_i$, $i=1$, $\ldots$, 6 and $H$ intersect non-trivially; $H_i$, $i=1$, $\ldots$, 6 are subgroups of $H$;
no two remaining subgroups have non-trivial intersection.
It follows that $\mathscr{I}(G)$ has a subgraph $\mathcal A_1$ as shown in Figure~\ref{fig:f2}, so $\gamma(\mathscr{I}(G))>1$ and $\overline{\gamma}(\mathscr{I}(G))>1$.

If $p\geq 7$, then $H:=\mathbb Z_p\times \mathbb Z_p$, $H_1:=\langle (1,0)\rangle$, $H_2:=\langle (1,1)\rangle$, $H_3:=\langle (2,1)\rangle$,
$H_4:=\langle (3,1)\rangle$, $H_5:=\langle (4,1)\rangle$, $H_6:=\langle (5,1)\rangle$, $H_7:=\langle (0,1)\rangle$, $H_8:=\mathbb Z_q\times \{e\}$, $H_iH_8$,
$i=1$, $\ldots$, 7 are proper subgroups of $G$. Here $H_8$ is a subgroup of $H_iH_8$,
$i=1$, $\ldots$, 7, so they form $K_8$ as a subgraph of $\mathscr{I}(G)$ and hence $\gamma(\mathscr{I}(G))>1$, $\overline{\gamma}(\mathscr{I}(G))>1$.

\noindent \textbf{Case 4:} $G\cong \mathbb Z_{p^2q}\times \mathbb Z_p$. Here $H:=\mathbb Z_{p^2}\times \mathbb Z_p$ is a subgroup of $G$.
As in Case 2,  $H$ has at least three proper subgroups of order $p^2$, say $H_i$, $i=1$, 2, 3  and has at least three subgroups of order $p$, say $H_i$, $i=4$, 5, 6; $H_j$, $j=4,5,6$ are subgroups of $H_i$ for some $i\in\{1$, 2, $3\}$, let it be $H_3$; also for some $j\in\{4$, 5, $6\}$, $H_j$ is a
subgroup of $H_i$, for every $i=1$, 2, 3,  let it be $H_4$. Let $H_7$ be a subgroup of $G$ of order $q$. Here
$H_iH_7$, $H_jH_7$, $i=1$, 2, 3 and $j=4$, 5, 6 are subgroups of $G$ and they have $H_7$ as their intersection, so $H_7$, $H_iK$, $H_jK$, $i=1$, 2,
$j=4$, 5, 6 form $K_6$ as a subgroup of $\mathscr{I}(G)$; $H$ and $H_3H_7$ intersect non-trivially and they also intersect with $H_jH_7$, $j=4$, 5, 6
non-trivially. It follows
that $\mathscr{I}(G)$ has a subgraph $\mathcal A_1$ as shown in Figure~\ref{fig:f2}, so $\gamma(\mathscr{I}(G))>1$ and $\overline{\gamma}(\mathscr{I}(G))>1$.

\noindent\textbf{Case 5:} $G \cong \mathbb Z_{p^2} \times \mathbb Z_{p^2}=\langle a,b~|~a^{p^2}=b^{p^2}=1, ab=ba\rangle$.
Then $H_1:=\langle a\rangle$, $H_2:=\langle ab^p\rangle$, $H_3:=\langle a^p, b^p\rangle$, $H_4:=\langle a,b^p\rangle$, $H_5:=\langle a^p, a^pb\rangle$,
$H_6:=\langle a^p\rangle$, $H_7:=\langle b\rangle$, $H_8:=\langle b^p\rangle$ are subgroups of $G$. Also $H_6$ is a subgroup of $H_i$, $i=1$, $\ldots$, 5;
$H_8$ is a subgroup of $H_3$, $H_4$, $H_5$, $H_7$. It follows that $\mathscr{I}(G)$ has a subgraph $\mathcal A_1$ as shown in Figure~\ref{fig:f2}, so $\gamma(\mathscr{I}(G))>1$ and $\overline{\gamma}(\mathscr{I}(G))>1$.

\noindent\textbf{Case 6:} $G \cong \mathbb Z_{p^k} \times \mathbb Z_{p^l}$, $k$, $l$ $\geq 3$. Then $\mathbb Z_{p^2} \times \mathbb Z_{p^2}$
is a proper subgroup of $G$ and so by Case 4, we have $\gamma(\mathscr{I}(G))>1$ and $\overline{\gamma}(\mathscr{I}(G))>1$.

\noindent \textbf{Case 7:} $G\cong \mathbb Z_p\times \mathbb Z_{pqr}$. Then $H:=\mathbb Z_p\times \mathbb Z_p$ is a subgroup of $G$ and
as in Case 1,
$H$ has at least three proper subgroups of order $p$, say $H_i$, $i=1,2,3$. Let $H_4$, $H_5$ be  subgroups of $G$ of orders $q$, $r$ respectively.
Here $H_iH_j$, $i=1$, 2, 3, $j=4$, 5 are also proper subgroups of $G$. So $H_iH_4$, $i=1$, 2, 3, $HH_4$, $H_4$ forms
$K_5$ as a subgraph of $\mathscr{I}(G)$; $H_iH_5$, $i=1$, 2, 3, $HH_5$, $H_5$ also forms $K_5$ as a subgraph of $\mathscr{I}(G)$. Thus $\mathscr{I}(G)$ has a subgraph $\mathcal B_1$ as shown in Figure~\ref{fig:int f3}, so $\gamma(\mathscr{I}(G))>1$.  $\mathcal B_1$ is also a topological obstruction for projective-plane embedding (e.g., see Theorem. 0.1 and the graph $A_5$ of case (3.15) on p. 343 in \cite{glov}), so $\overline{\gamma}(\mathscr{I}(G))>1$.

\noindent\textbf{Case 8:} $G \cong\mathbb Z_p\times\mathbb Z_p\times\mathbb Z_p$. Here $\langle(1,0,0)$, $(0,1,0)\rangle$, $\langle(1,0,0)$, $(0,0,1)\rangle$,
$\langle(1,0,0)$, $(1,1,0)\rangle$, $\langle(1,0,0)$, $(0,1,1)\rangle$ are proper subgroups of $G$, which have $\langle(1,0,0)\rangle$ as their intersection.
It follows that these five subgroups form $K_5$ as a subgraph of $\mathscr{I}(G)$.
Also $\langle(1,1,1)$, $(0,1,0)\rangle$, $\langle(0,1,0)$, $(0,0,1)\rangle$,
$\langle(0,1,0)$, $(1,1,0)\rangle$, $\langle(0,1,0)$, $(0,1,1)\rangle$ are proper subgroups of $G$, which have $\langle(0,1,0)\rangle$ as their intersection
 and so they also form $K_5$ as a subgraph of $\mathscr{I}(G)$. Thus $\mathscr{I}(G)$ has  a subgraph $\mathcal B_1$ as shown in Figure~\ref{fig:int f3}, so $\gamma(\mathscr{I}(G))>1$ and $\overline{\gamma}(\mathscr{I}(G))>1$.

\noindent\textbf{Case 9:} $G \cong {\mathbb Z}_{p_1^{\alpha_1}} \times {\mathbb Z}_{p_2^{\alpha_2}} \times \ldots \times {\mathbb Z_{p_k^{\alpha_k}}}$, where
$p_i$'s are primes with at least two $p_i$'s are equal and $\alpha_i\geq1$. If $k\geq 2$, then $G$ has one of the following groups as its subgroup: $\mathbb Z_{p_i^2}\times \mathbb Z_{p_i^2}$,
$\mathbb Z_{p_i^2p_j}\times \mathbb Z_{p_i}$, $\mathbb Z_{p_i}\times \mathbb Z_{p_i}\times \mathbb Z_{p_i}$, for some $i$, $j$.
So by Cases 3, 4, 7, the intersection graph of subgroups of these subgroups are non-toroidal, non-projective planar and contains $K_5$. So it follows that $\gamma(\mathscr{I}(G))>1$, $\overline{\gamma}(\mathscr{I}(G))>1$ and $\mathscr{I}(G)$ contains $K_5$ as a subgraph.

The result follows by combining all the above cases.
\end{proof}

\section{Finite non-abelian groups}\label{sec:4}
In this section, we classify the finite non-abelian groups whose intersection graphs of subgroups are one of toroidal, projective-planar, $K_5$ free, $C_3$-free, acyclic or bipartite. We first investigate the non-abelian solvable groups and then we deal with the nonsolvable groups.

For any integer $n\geq 3$, the dihedral group of order $2n$ is given by $D_{2n}=\langle a, b~|~a^n=b^2=1, ab=ba^{-1}\rangle$;
For any integer $n \geq 2$, the generalized quaternion group of order $2^n$ is given by
$Q_{2^n} = \big < a, b ~|~a^{2^{n-1}} = b^4 = 1, a^{2^{n-2}} = b^2 = 1,
bab^{-1} = a^{-1}\big >$; For any $\alpha\geq 3$ and $p$ is a prime, the modular group of order $p^\alpha$ is given by
$M_{p^\alpha}=\langle a,b~|~a^{p^{\alpha-1}}=b^p=1, bab^{-1}=a^{p^{\alpha-2}+1}\rangle$; $S_n$ and $A_n$ are symmetric and alternating groups of degree $n$
acting on $\{1,2, \ldots, n\}$ respectively.
We denote the order of an element $a \in \mathbb Z_n$ by $\text{ord}_n(a)$. The number of Sylow $p$-subgroups of a group $G$
is denoted by $n_p(G)$.
\subsection{Finite non-abelian solvable groups}\label{sec:4a}
\begin{pro}\label{intersecting graph t2}
 Let $G$ be a non-abelian group of order $p^{\alpha}$, where $p$ is a prime and $\alpha \geq 3$. Then
\begin{enumerate}[{\normalfont (1)}]
\item $\mathscr{I}(G)$ is $K_5$-free if and only if $G$ is  either $Q_8$ or $M_8$;

\item $\mathscr{I}(G)$ contains $C_3$;
\item $\mathscr{I}(G)$ is toroidal if and only if $G$ is  one of $M_{p^3}(p=3$, $5)$ or $M_{16}$.
\item $\mathscr{I}(G)$ is projective-planar if and only if $G\cong M_{27}$.
\end{enumerate}
\end{pro}
\begin{proof} We prove the result in the following cases:

\noindent\textbf{Case 1:} $\alpha =3$. If $p=2$, then the only non-abelian groups of order 8 are $Q_8$ and $M_8$. By Theorem~\ref{1000}, the intersection graphs of subgroups of these two groups are planar and by \eqref{23} and Figure~\ref{int f1}(c), they contains $C_3$.

If $p \neq 2$, then up to isomorphism the only non-abelian groups of order $p^3$ are $M_{p^3}$ and $(\mathbb Z_p \times \mathbb Z_p) \rtimes \mathbb Z_p$.
\begin{enumerate}[{\normalfont (i)}]
\item If $G \cong M_{p^3}$, then the subgroup lattices of $M_{p^3}$ and $\mathbb Z_{p^2}\times \mathbb Z_p$ are isomorphic, so their intersection graphs of subgroups are also isomorphic.
By Case 2 in the proof of Proposition~\ref{intersecting graph t1}, we have
\begin{equation}\label{e100}
\mathscr{I}(G)\cong K_1+(K_{p+1}\cup \overline{K}_p).
\end{equation}
 Also, $\gamma(\mathscr{I}(G))=1$ if and only if $p=3,5$; $\overline{\gamma}(\mathscr{I}(G))=1$ if only if $p=3$; $\mathscr{I}(G)$ contains $K_5$.

\item If $G \cong (\mathbb Z_p \times \mathbb Z_p) \rtimes \mathbb Z_p$, then consider
its proper subgroups $H_1:=\langle a, b \rangle$, $H_2:=\langle a,c \rangle$, $H_3:=\langle ab, c \rangle$, $H_4:=\langle ab^2, c\rangle$,
$H_5:=\langle b, ac \rangle$, $H_6:=\langle b, a^2c\rangle$, $H_7:=\langle b, c\rangle$, $H_8:=\langle b\rangle$, $H_9:=\langle c\rangle$.
Here $H_9$ is a subgroup of $H_2$, $H_3$, $H_4$, $H_7$, so these five subgroups forms $K_5$ as a subgraph of  $\mathscr{I}(G)$; $H_8$ is a subgroup of
$H_1$, $H_5$, $H_6$, so these four subgroups forms $K_4$ as a proper subgraph of $\mathscr{I}(G)$;
$H_1$, $H_4$ intersects non-trivially; $H_8$ is a subgroup of $H_7$, $H_5$, $H_6$. It follows that
$\mathscr{I}(G)$ has a subgraph $\mathcal B_3$ as shown in Figure~\ref{fig:int f3}, so $\gamma(\mathscr{I}(G))>1$.

Also $\mathscr{I}(G)$ has a
subgraph as shown in Figure~\ref{fig:int f8}, which is an obstruction for projective-plane (e.g., see Theorem 0.1 and graph
$C_7$ of case (3.20) on p. 344 in \cite{glov}). Therefore, $\overline{\gamma}(\mathscr{I}(G))>1$.
\begin{figure}[ ht ]
\begin{center}
\includegraphics[scale=.8]{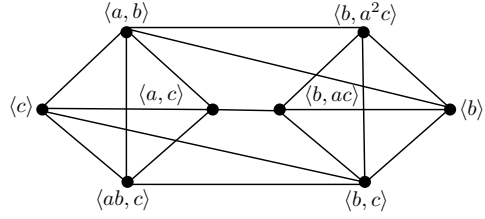}
\caption{An obstruction for the  projective-plane.}
\label{fig:int f8}
\end{center}
\end{figure}
\end{enumerate}
\noindent\textbf{Case 2:} $\alpha=4$. According to  Burnside~\cite{burn}, up to isomorphism, there are fifteen groups of order $p^4$, $p>2$ and there are nine groups
of order $2^4$.
In following we investigate the genus of the intersection graphs of subgroups of each of them:

\noindent \textbf{Sub case 2a:} $p>2$.

\begin{enumerate}[{\normalfont (i)}]
\item $G\cong M_{p^4}$. Here $H_1:=\langle a\rangle$, $H_2:=\langle ab\rangle$, $H_3:=\langle ab^2\rangle$, $H_4:=\langle a^p\rangle$,
$H_5:=\langle a^pb\rangle$, $H_6:=\langle a^pb^2\rangle$, $H_7:=\langle a^p, b\rangle$, $H_8:=\langle a^{p^2}\rangle$ are proper subgroups of $G$. Also
$H_8$ is a subgroup of $H_i$, $i=1$, $\ldots$, 7. It follows that they form $K_8$ as a
subgraph of $\mathscr{I}(G)$, so $\gamma(\mathscr{I}(G))>1$.

\item $G\cong \langle a,b,c~|~a^{p^2}=b^p=c^p=1, cb=a^pbc, ab=ba, ac=ca\rangle$. Here $H_1:=\langle a,b\rangle$,
$H_2:=\langle a,c\rangle$, $H_3:=\langle a\rangle$, $H_4:=\langle a^p, c\rangle$, $H_5:=\langle a^p,b\rangle$,
$H_6:=\langle ab\rangle$, $H_7:=\langle ac\rangle$, $H_8:=\langle a^p\rangle$ are proper subgroups of $G$. Also $H_8$ is a subgroup of $H_i$, $i=1$, $\ldots$, 7.
It follows that they form $K_8$ as a subgraph of $\mathscr{I}(G)$, so $\gamma(\mathscr{I}(G))>1$.

\item $G\cong \langle a,b~|~a^{p^2}=b^{p^2}=1, bab^{-1}=a^{1+p}\rangle$. Here $H_1:=\langle a\rangle$, $H_2:=\langle a^p\rangle$, $H_3:=\langle a^p,b\rangle$,
$H_4:=\langle a^p, b^p\rangle$, $H_5:=\langle a, b^p\rangle$, $H_6:=\langle b\rangle$, $H_7:=\langle b^p\rangle$, $H_8:=\langle ab, b^p\rangle$,
$H_9:=\langle a^2b, b^p\rangle$ are proper subgroups of $G$. Here $H_i$, $i=1$, $\ldots$, 5 intersect with each other non-trivially
and so they form $K_5$ as a subgraph of
$\mathscr{I}(G)$; $H_7$ is a subgroup of $H_4$, $H_6$, $H_8$, $H_9$; $H_6$, $H_8$, $H_9$  intersect with $H_5$ non-trivially.
It follows that
$\mathscr{I}(G)$ has a subgraph $\mathcal B_3$ as shown in Figure~\ref{fig:int f3}, so $\gamma(\mathscr{I}(G))>1$.

\item $G\cong \langle a,b,c~|~a^{p^2}=b^p=c^p=1, ca=a^{1+p}c, ab=ba, cb=bc\rangle$. We can take the proper subgroups of $G$ as in (ii), except by taking
$H_7:=\langle a, bc\rangle$ instead of $H_7=\langle ac\rangle$. Then
$\mathscr{I}(G)$ has $K_8$ as a subgraph and so $\gamma(\mathscr{I}(G))>1$.

\item $G\cong \langle a,b,c~|~a^{p^2}=b^p=c^p=1, ca=abc, ab=ba, cb=bc\rangle$. We can use a similar argument as in (iv) to show
$\mathscr{I}(G)$ has $K_8$ as a subgraph and so $\gamma(\mathscr{I}(G))>1$.

\item $G\cong \langle a,b,c~|~a^{p^2}=b^p=c^p=1, ba=a^{1+p}b, ca=abc, cb=bc\rangle$. Here $H_1:=\langle a,b\rangle$, $H_2:=\langle a,c\rangle$,
$H_3:=\langle b,c\rangle$,
$H_4:=\langle a^p, c\rangle$, $H_5:=\langle a^p,b\rangle$, $H_6:=\langle a^p\rangle$, $H_7:=\langle a\rangle$, $H_8:=\langle a^p, bc\rangle$,
$H_9:=\langle a^p, b^2c\rangle$ are proper subgroups of $G$. Also $H_i$, $i=1$, $\ldots$, 5 intersect with each other non-trivially
and so they form $K_5$ as a subgraph of
$\mathscr{I}(G)$; $H_6$ is a subgroup of $H_7$, $H_8$, $H_9$; $H_6$ is a subgroup of $H_1$; $H_7$, $H_8$, $H_9$ intersect with $H_4$ non-trivially.
It follows that
$\mathscr{I}(G)$ has a subgraph $\mathcal B_3$ as shown in Figure~\ref{fig:int f3}, so $\gamma(\mathscr{I}(G))>1$.

\item If $p=3$, then  $G\cong \langle a,b,c~|~a^{p^2}=b^p=c^{p^2}=1, c^p=a^p, ab=ba^{1+p}, ac=cab^{-1}, cb=bc\rangle$. Here
$H_1:=\langle a,b\rangle$, $H_2:=\langle a,c\rangle$, $H_3:=\langle b,c\rangle$,
$H_4:=\langle a^p, b\rangle$, $H_5:=\langle c\rangle$, $H_6:=\langle a\rangle$, $H_7:=\langle c^p\rangle$, $H_8:=\langle bc\rangle$,
$H_9:=\langle b^2c\rangle$ are subgroups of $G$. Also $H_i$, $i=1$, $\ldots$, 5 intersect with each other non-trivially and so they form $K_5$ as a subgraph of
$\mathscr{I}(G)$; $H_7$ is a subgroup of $H_6$, $H_8$, $H_9$; $H_7$ is a subgroup of $H_4$; $H_6$, $H_8$, $H_9$ intersect with $H_5$ non-trivially.
It follows that
$\mathscr{I}(G)$ has a subgraph $\mathcal B_3$ as shown in Figure~\ref{fig:int f3},
so $\gamma(\mathscr{I}(G))>1$.

If $p>3$, then $G\cong \langle a,b,c~|~a^{p^2}=b^p=c^{p^2}=1, ba=a^{1+p}b, ca=a^{1+p}bc, cb=a^pbc\rangle$.
Here $\langle a,b\rangle\cong M_{p^3}$ is a subgroup of $G$,
so by \eqref{e100}, $M_{p^3}$ together with its proper subgroups forms $K_8$ as a subgraph of $\mathscr{I}(G)$. Hence $\gamma(\mathscr{I}(G))>1$.

\item If $p=3$, then $G\cong \langle a,b,c~|~a^{p^2}=b^p=c^{p^2}=1, c^p=a^{-p}, ab=ba^{1+p}, ac=cab^{-1}, cb=bc\rangle$.
We can use a similar argument as in (vii), to show $\gamma(\mathscr{I}(G))>1$.

If $p>3$, then $G\cong \langle a,b,c~|~a^{p^2}=b^p=c^{p^2}=1, ba=a^{1+p}b, ca=a^{1+dp}bc, cb=a^{dp}bc, d\ncong 0,1 (\mbox {mod} p)\rangle$.
Here $\langle a,b\rangle\cong M_{p^3}$ is a subgroup of $G$,
so by \eqref{e100}, $M_{p^3}$ together with its proper subgroups forms $K_8$ as a subgraph of $\mathscr{I}(G)$. Thus $\gamma(\mathscr{I}(G))>1$.

\item $G\cong \langle a,b,c,d~|~a^p=b^p=c^p=d^p=1, dc=acd, bd=db, ad=da, bc=cb, ac=ca\rangle$. Here $\langle a,b,c\rangle\cong
\mathbb Z_p\times \mathbb Z_p\times \mathbb Z_p$ is a subgroup of $G$ and so by Case 7 in the proof of Proposition~\ref{intersecting graph t1},
it follows that $\gamma(\mathscr{I}(G))>1$.

\item If $p=3$, then $G\cong \langle a,b,c~|~a^{p^2}=b^p=c^p=1, ab=ba, ac=cab, cb=ca^{-p}b\rangle$.
Here $H_1:=\langle a,b\rangle$, $H_2:=\langle a,c\rangle$, $H_3:=\langle b,c\rangle$,
$H_4:=\langle a^p, b\rangle$, $H_5:=\langle a^p, c\rangle$, $H_6:=\langle a\rangle$, $H_7:=\langle a^p\rangle$, $H_8:=\langle ab\rangle$,
$H_9:=\langle a^2b\rangle$ are proper subgroups of $G$. Also $H_i$, $i=1$, $\ldots$, 5 intersect with each other non-trivially
and so they form $K_5$ as a subgraph of
$\mathscr{I}(G)$; $H_7$ is a subgroup of $H_4$, $H_6$, $H_8$, $H_9$; $H_6$, $H_8$, $H_9$ intersect with $H_5$ non-trivially.
It follows that
$\mathscr{I}(G)$ has a subgraph $\mathcal B_3$ as shown in Figure~\ref{fig:int f3}, so $\gamma(\mathscr{I}(G))>1$.

If $p>3$, then $G\cong \langle a,b,c,d~|~a^p=b^p=c^p=d^p=1, dc=acd, bd=db, ad=da, bc=cb, ac=ca\rangle$. Here $\langle a,b,c\rangle\cong
\mathbb Z_p\times \mathbb Z_p\times \mathbb Z_p$ is a subgroup of $G$ and so by Case 7 in the proof of Proposition~\ref{intersecting graph t1},
it follows that $\gamma(\mathscr{I}(G))>1$.
\end{enumerate}
\noindent \textbf{Sub case 2b:} $p=2$.
\begin{enumerate}[{\normalfont (i)}]
\item  $G\cong M_{2^4}$. Here $H_1:=\langle a\rangle$, $H_2:=\langle ab\rangle$, $H_3:=\langle a^2,b\rangle$,
$H_4:=\langle a^2\rangle$, $H_5:=\langle a^2b, c\rangle$, $H_6:=\langle a^4,b\rangle$, $H_7:=\langle a^4\rangle$, $H_8:=\langle a^4b\rangle$,
$H_9:=\langle b\rangle$ are the only subgroups of $M_{2^4}$. Further $H_7$ is a subgroup of $H_i$, $i=1$, $\ldots$, 6;
$H_8$, $H_9$ are proper subgroups of $H_6$;
no two remaining subgroups intersect non-trivially. It follows that
\begin{equation}\label{e101}
\mathscr{I}(G)\cong K_1+(K_6\cup \overline{K}_2).
\end{equation}
Note that $\mathscr{I}(G)$ is a graph obtained by attaching 2 pendent edges to any one of the vertices of $K_7$, so $\gamma(\mathscr{I}(G))=1$.

\item $G\cong \langle a,b,c~|~a^4=b^4=c^2=1, bab^{-1}=a^{-1}, a^2=b^2, bc=cb, ac=ca\rangle$. Here
$H_1:=\langle a, c\rangle$, $H_2:=\langle a\rangle$, $H_3:=\langle ac\rangle$,
$H_4:=\langle a^2\rangle$, $H_5:=\langle a^2, c\rangle$, $H_6:=\langle bc\rangle$, $H_7:=\langle b\rangle$, $H_8:=\langle b, c\rangle$
are proper subgroups of
$G$ and these subgroups intersect with each other non-trivially, so they form $K_8$ as a subgraph of $\mathscr{I}(G)$. It follows that $\gamma(\mathscr{I}(G))>1$.

\item $G\cong \langle a,b~|~a^8=b^2=1, bab^{-1}=a^{-1}\rangle$. Here
$H_1:=\langle a\rangle$, $H_2:=\langle a^2\rangle$, $H_3:=\langle a^4\rangle$,
$H_4:=\langle a^2, b\rangle$, $H_5:=\langle a^4, b\rangle$, $H_6:=\langle a^4, ab\rangle$, $H_7:=\langle a^4, a^2b\rangle$, $H_8:=\langle a^2, ab\rangle$
are proper subgroups of
$G$ and these eight subgroups intersect with each other non-trivially, so they form $K_8$ as a subgraph of $\mathscr{I}(G)$. It follows that $\gamma(\mathscr{I}(G))>1$.

\item $G\cong \langle a,b~|~a^8=b^2=1, b^{-1}ab=a^3\rangle$. Here
$H_1:=\langle a\rangle$, $H_2:=\langle a^2\rangle$, $H_3:=\langle a^4\rangle$,
$H_4:=\langle a^2, b\rangle$, $H_5:=\langle a^4, b\rangle$, $H_6:=\langle a^4, ab\rangle$, $H_7:=\langle a^4, a^2b\rangle$, $H_8:=\langle a^4, a^3b\rangle$
are proper subgroups of
$G$ and these eight subgroups intersect with each other non-trivially. Hence they form $K_8$ as a subgraph of $\mathscr{I}(G)$.
 It follows that $\gamma(\mathscr{I}(G))>1$.

\item $G\cong \langle a,b~|~a^8=b^2=1, b^{-1}ab=a^{-1}, b^2=a^4\rangle$. Here
$H_1:=\langle a\rangle$, $H_2:=\langle a^2\rangle$, $H_3:=\langle a^4\rangle$,
$H_4:=\langle b\rangle$, $H_5:=\langle a^2, b\rangle$, $H_6:=\langle ab\rangle$, $H_7:=\langle a^2b\rangle$, $H_8:=\langle a^3b\rangle$
are proper subgroups of
$G$ and these eight subgroups intersect with each other non-trivially, so they form $K_8$ as a subgraph of $\mathscr{I}(G)$. It follows that $\gamma(\mathscr{I}(G))>1$.
\end{enumerate}
 The remaining groups in this subcase are identical with the groups described in (ii), (iii), (iv), (v) of Subcase 2a.

\noindent\textbf{Case 3:} $\alpha=5$. By Theorem~\ref{501} (i),
$G$ has at least three subgroups of order $p^4$, say $H_i$, $i=1, 2, 3$  and at least three subgroups of order $p^3$, say
$H_i$, $i=4, 5, 6$. Here for each $i=1$, 2, $3$ and $j=4$, 5, $6$, $H_i$ and $H_j$ have a non-trivial intersection. For otherwise, $|H_iH_j|=p^k$, $k=7$ or
8, which is not possible. Let $H_7$ be a common  subgroup of order $p^2$ for both $H_4$ and $H_5$. Let $H_8$ be a subgroup of $H_7$ of order $p$.
Here $H_1$, $H_4$, $H_5$, $H_7$, $H_8$
intersects with each other non-trivially, so they form $K_5$ as a subgraph of $\mathscr{I}(G)$.
Now $H_2$ and $H_3$ have a common subgroup of order
$p^3$, say $H_9$; let $H_{10}$ be a subgroup of $H_9$ of order $p^2$; let $H_{11}$ be a subgroup of $H_{10}$ of order $p$.
Then $H_2$, $H_3$, $H_9$, $H_{10}$, $H_{11}$ intersect with each other non-trivially.
Therefore, $\mathscr{I}(G)$ contains a subgraph $\mathcal B_1$ as shown in Figure~\ref{fig:int f3}, so $\gamma(\mathscr{I}(G))>1$.

\noindent\textbf{Case 4:} $\alpha\geq6$. By Theorem~\ref{501} (i), (ii),
$G$ has at least two subgroups of order $p^{\alpha-1}$, say $H_1$, $H_2$. Let $H_i$, $i=3$, 4, 5, 6 be subgroups of $H_1$ of order $p^{\alpha-2}$,
$p^3$, $p^2$, $p$ respectively. Also let $H_i$, $i=7$, 8, 9, 10 be subgroups of $H_2$ of orders $p^{\alpha-2}$, $p^3$, $p^2$, $p$ respectively.
It follows that $H_1$, $H_3$, $H_4$, $H_5$, $H_6$ forms $K_5$ as a subgraph of  $\mathscr{I}(G)$ and $H_2$, $H_7$, $H_8$, $H_9$, $H_{10}$ forms another copy of $K_5$ as a subgraph of $\mathscr{I}(G)$.
Thus $\mathscr{I}(G)$ has a subgraph $\mathcal B_1$ as shown in Figure~\ref{fig:int f3}, so $\gamma(\mathscr{I}(G))>1$.

Now we investigate the projective-plane embedding of the intersection graphs of subgroups of $G$, when $\alpha\geq 4$. We need to consider the following cases:

\noindent\textbf{Case a:} $p>2$. Then by Theorem~\ref{501} (i),
$G$ has at least four subgroups of order $p^3$, let them be $H_i$, $i=1, 2, 3, 4$; and has at least
four subgroups of order $p^2$, say $B_i$, $i=1, 2, 3, 4$. Here $H_i\cap B_j\neq \emptyset$ for all $i,j=1, 2, 3, 4$ with $i \neq j$  and so they form $K_{4,4}$ as a subgraph of $\mathscr{I}(G)$ with
bipartition $X:=\{H_1$, $H_2$, $H_3$, $H_4\}$ and $Y:=\{B_1$, $B_2$, $B_3$, $B_4\}$. It follows that
$\overline{\gamma}(\mathscr{I}(G))>1$.

\noindent\textbf{Case b:} $p=2$. If $G$ has an unique subgroup of order $2$, then by Theorem~\ref{501} (i), (ii), $G\cong Q_{2^\alpha}$. Then $G$ has at least seven proper subgroups and they have a unique subgroup of order 2 in common. It follows that
$\mathscr{I}(G)$ has $K_7$ as a subgraph, so $\overline{\gamma}(\mathscr{I}(G))>1$. If $G\neq Q_{2^\alpha}$, then it has atleast three subgroups of order $p^{\alpha-1}$,
let them be $H_1$, $H_2$, $H_3$;
$G$ has atleast three subgroups, say $H_4$, $H_5$, $H_6$ of order $p^{\alpha-2}$; atleast two subgroups of order $p^{\alpha-3}$, say $H_7$, $H_8$. Here
$H_i\cap H_j\neq \emptyset$, for every $i=1$, 2, 3 and $j=4$, 5, 6, $i\neq j$; $H_7$, $H_8$ are subgroups of $H_i$, $i=1$, 2, 3. It follows that
$K_{3,5}$ is a subgraph of $\mathscr{I}(G)$ with bipartition $X:=\{H_1$, $H_2$, $H_3\}$ and $Y:=\{H_4$, $H_5$, $H_6$, $H_7$, $H_8\}$, so
$\overline{\gamma}(\mathscr{I}(G))>1$.

The proof follows by putting together all the cases.
\end{proof}

 If $G$ is the non-abelian group of order $pq$, where $p< q$ and $p$, $q$ are two distinct primes,
then by Theorem~\ref{1000}, $\gamma(\mathscr{I}(G))$ is planar and by \eqref{e8} it is acyclic.

Next we investigate the groups of order greater than $pq$.

Consider the semi-direct product $\mathbb Z_q \rtimes_{t} \mathbb Z_{p^{\alpha}} = \langle a,b | a^q= b^{p^{\alpha}}= 1, bab^{-1}= a^i,
{ord_{q}}(i)= p^t \rangle$, where $p$, $q$ are distinct primes with $p^t~|~(q-1)$, $t \geq 0$.
Then every semi-direct product $Z_q \rtimes Z_{p^{\alpha}}$ is  one of these types \cite[Lemma 2.12]{boh-reid}. So here after, when $t = 1$ we will suppress the subscript.

\begin{pro}\label{intersecting graph t4}
Let $G$ be a non-abelian group of order $p^2q$, where $p$ and $q$ are distinct primes. Then
\begin{enumerate}[{\normalfont (1)}]
\item $\mathscr{I}(G)$ is $K_5$-free if and only if $G$ is  either $\mathbb Z_q \rtimes_{2} \mathbb Z_{p^2}$, $\langle a, b, c~|~ a^p=b^p=c^q=1, ab=ba, cac^{-1}=b,
cbc^{-1}= a^1b^{l} \rangle$, where $\bigl(\begin{smallmatrix}
  0 & -1\\ 1 & l
\end{smallmatrix} \bigr)$ has order $q$ in $GL_2(p)$, $q|(p+1)$ or $A_4$.
\item The following are equivalent:
\begin{enumerate}[\normalfont (a)]
\item $G\cong \langle a, b, c~|~ a^p=b^p=c^q=1, ab=ba, cac^{-1}=b,
cbc^{-1}= a^{-1}b^{l} \rangle$, where $\bigl(\begin{smallmatrix}
  0 & -1\\ 1 & l
\end{smallmatrix} \bigr)$ has order $q$ in $GL_2(p)$, $q|(p+1)$ or $A_4$;
\item $\mathscr{I}(G)$ is $C_3$-free;
\item $\mathscr{I}(G)$ is acyclic;
\item $\mathscr{I}(G)$ is bipartite.
\end{enumerate}
\item $\mathscr{I}(G)$ is toroidal if and only if $G$ is  one of $\mathbb Z_3\rtimes \mathbb Z_4$, $\mathbb Z_5\rtimes \mathbb Z_4$ or
$\mathbb Z_9\rtimes \mathbb Z_2$, $\mathbb Z_{25}\rtimes \mathbb Z_2$.
\item $\mathscr{I}(G)$ is projective-planar if and only if $G$ is  either $\mathbb Z_3\rtimes \mathbb Z_4$ or $\mathbb Z_9\rtimes \mathbb Z_2$.
\end{enumerate}
\end{pro}
\begin{proof}
To prove the result we use the classification of groups of order $p^2q$ given in \cite[p.~76-80]{burn}. We have the following cases to consider.

\noindent\textbf{Case 1:} $p<q$:

\noindent\textbf{Case 1a:} $p \nmid (q-1)$. By Sylow's Theorem, there is no non-abelian group in this case.

\noindent\textbf{Case 1b:} $p \mid (q-1)$ but $p^2 \nmid (q-1)$. In this case, there are two non-abelian groups.

The first group is $G_1:= \mathbb Z_q \rtimes \mathbb Z_{p^2} = \langle a, b | a^q= b^{p^2}=1, bab^{-1}=a^i, ord_q(i)=p \rangle$.
We have $\langle a\rangle$, $\langle a^ib\rangle$, $i=1,2,\ldots,q$, $\langle b^p\rangle$ and
$\langle ab^p\rangle$ are the only proper subgroups of $G_1$. Here $\langle b^p\rangle$ is a subgroup of the remaining proper subgroups, except
$\langle a\rangle$. Also $\langle a\rangle$ is a subgroup of $\langle ab^p\rangle$. It follows that
\begin{equation}\label{e1}
 \mathscr{I}(G_1)\cong K_1+(K_1\cup K_{q+1}).
\end{equation}

Note that $q=5$ is not possible here, since $p=2$ is such that $p \mid (q-1)$ but $p^2 \mid (q-1)$.
Note that $\mathscr{I}(G)$ is a graph obtained by attaching one pendent edge to any one of the vertices of $K_{q+2}$. So  $\gamma(\mathscr{I}(G))=1$ if and only if $q=3$; $\overline{\gamma}(\mathscr{I}(G))=1$ if and only if $q=3$; $\mathscr{I}(G)$ contains $K_5$.

The second group is $G_2:= \langle a, b, c| a^q= b^p= c^p=1, bab^{-1}=a^i, ac=ca, bc=cb, ord_q(i)=p \rangle$. Here $H_1:=\langle a,b\rangle$,
$H_2:=\langle a,c\rangle$, $H_3:=\langle b,c\rangle$, $H_4:=\langle ab,c\rangle$, $H_5:=\langle a^2b,c\rangle$, $H_6:=\langle a^3b,c\rangle$,
$H_7:=\langle c\rangle$, $H_8:=\langle a^4b,c\rangle$ are proper subgroups of $G$. Also $H_i$, $i=3,4$,$\ldots , 8$ intersect with each other non-trivially
and so they form $K_6$
as a subgraph of $\mathscr{I}(G_2)$; $H_1$ and $H_2$ intersect non-trivially; also they intersect non-trivially with $H_3$, $H_4$, $H_5$.
It follows that $\mathscr{I}(G_2)$ has a
subgraph $\mathcal A_1$ as shown in Figure~\ref{fig:f2}, so $\gamma(\mathscr{I}(G_2))>1$ and $\overline{\gamma}(\mathscr{I}(G_2))>1$.

\noindent\textbf{Case 1c:} $p^2 | (q-1)$. In this case, we have both groups $G_1$ and $G_2$ from Case 1b together with the group $G_3:= \mathbb Z_q \rtimes_2 \mathbb Z_{p^2} = \langle a, b | a^q= b^{p^2}=1, bab^{-1}=a^i, ord_q(i)=p^2 \rangle$. Note that here $q=5$, $p=2$ is possible for the group  $G_1$. By~\eqref{e1}, $\gamma(\mathscr{I}(G_1))=1$, $\overline{\gamma}(\mathscr{I}(G_1))>1$.
$\mathscr{I}(G_2)$ is already discussed in Case 1b.  By Theorem~\ref{1000}, $\mathscr{I}(G_3)$ is planar and by \eqref{e9}, it contains $C_3$.

\noindent\textbf{Case 2:} $p > q$

\noindent\textbf{Case 2a:} $q \nmid (p^2 -1)$. Then there is no  non-abelian subgroups.

\noindent\textbf{Case 2b:} $q | (p-1)$. In this case, we have two groups.

The first one is $G_4:=\langle a,b | a^{p^2}= b^q=1, bab^{-1}, ord_{p^2}(i)=q \rangle$. By Sylow's Theorem, $G_4$ has a unique subgroup say $N$,
of order $p^2$ and has a unique subgroup $N'$ of order $p$; $G_4$ has $p^2$ Sylow $q$-subgroups of order $q$, say
$H_i$, $i=1,2\ldots,p^2$; it has $p$ subgroups of order $pq$, say $A_i=\langle a^p,a^ib\rangle,i=0,1,\ldots p-1$. These are the
only proper subgroups of $G_4$. Here $N'$ is a subgroup
of $N$ and $A_i$, $i=0,1,\ldots,p-1$; $H_i,H_{i+p},\ldots,H_{i+p(p-1)}$ are subgroups of $A_i$ for each $i=0,1,\ldots,p-1$. It follows that
$\mathscr{I}(G_4)$ is the graph obtained by attaching a single pendant edge to any $p$ vertices of $K_{p+2}$.
Thus, $\gamma(\mathscr{I}(G_4))=1$ if and only if $p=3$, $5$; $\overline{\gamma}(\mathscr{I}(G_4))=1$ if and only if $p=3$;$\mathscr{I}(G_4)$ contains $K_5$.

Next we have the family of groups $\langle a, b, c | a^p=b^p=c^q=1, cac^{-1}=a^i, cbc^{-1}=b^{i^t}, ab=ba, ord_p(i)=q \rangle$. There are $(q+3)/2$
 isomorphism types in this family (one for $t=0$ and one for each pair $\{ x, x^{-1} \}$ in ${F_p}^{\times}$. We will refer to all of these groups as $G_{5(t)}$
of order $p^2q$. Here $H_1:=\langle a,b\rangle$,
$H_2:=\langle a,c\rangle$, $H_3:=\langle a,bc\rangle$, $H_4:=\langle a,b^2c\rangle$, $H_5:=\langle ab, c\rangle$,
$H_6:=\langle b,c\rangle$, $H_7:=\langle a^2b, c\rangle$, $H_8:=\langle c\rangle$ are proper subgroups of $G$. Further $H_i$, $i=1$, $\ldots$, 6 intersect non-trivially with each other and so they form $K_6$
as a subgraph of $\mathscr{I}(G_{5(t)})$; $H_7$, $H_8$ intersect non-trivially and they intersect with $H_2$, $H_5$, $H_6$.
It follows that $\mathscr{I}(G_{5(t)})$ has a subgraph $\mathcal A_1$ as shown in Figure~\ref{fig:f2},
so $\gamma(\mathscr{I}(G_{5(t)}))>1$, $\overline{\gamma}(\mathscr{I}((G_{5(t)}))$ and $\mathscr{I}(G_{5(t)})$ contains $K_5$ as a subgraph.

\noindent\textbf{Case 2c:} $q| (p+1)$. In this case, we have only one group of order $p^2q$, given by $G_6:= (\mathbb Z_p \times \mathbb Z_p) \rtimes \mathbb Z_q
= \langle a, b,c| a^p=b^p=c^q=1, ab=ba, cac^{-1}=a^ib^j, cbc^{-1}=a^kb^l\rangle$, where $\bigl(\begin{smallmatrix}
  i & j\\ k & l
\end{smallmatrix} \bigr)$ has order $q$ in $GL_2(p)$.
\begin{itemize}
\item [(i) ] If $G_6$ has a subgroup of order $pq$, then $H_1:=\langle a,b\rangle$,
$H_2:=\langle a,c\rangle$, $H_3:=\langle a,bc\rangle$, $H_4:=\langle a,b^2c\rangle$, $H_5:=\langle ab, c\rangle$,
$H_6:=\langle b,c\rangle$, $H_7:=\langle a^2b, c\rangle$, $H_8:=\langle c\rangle$ are proper subgroups of $G$. Also $H_i$, $i=1$, $\ldots$, 6 intersect non-trivially with each other and so they form $K_6$
as a subgraph of $\mathscr{I}(G)$; $H_7$, $H_8$ intersect non-trivially and they intersect with $H_2$, $H_5$, $H_6$. Thus
$\mathscr{I}(G_6)$ has a subgraph $\mathcal A_1$ as shown in Figure~\ref{fig:f2}, so $\gamma(\mathscr{I}(G_6))>1$ and
$\overline{\gamma}(\mathscr{I}(G))>1$.
\item [(ii) ]If $G_6$ has no subgroup of order $pq$, then $G_6:=\langle a, b, c~|~ a^p=b^p=c^q=1, ab=ba, cac^{-1}=b,
cbc^{-1}= a^{1}b^{l} \rangle$, where $\bigl(\begin{smallmatrix}
  0 & -1\\ 1 & l
\end{smallmatrix} \bigr)$ has order $q$ in $GL_2(p)$. By Theorem~\ref{1000}, $\mathscr{I}(G_6)$ is planar and by \eqref{e11}, it is acyclic.
\end{itemize}

Note that if $(p, q)= (2, 3)$, the Cases 1 and 2 are not mutually exclusive. Up to isomorphism, there are three non-abelian groups of order 12:
$\mathbb Z_3 \rtimes \mathbb Z_4$, $D_{12}$ and $A_4$. Here the intersection graph of subgroups of $\mathbb Z_3 \rtimes \mathbb Z_4$ (the group $G_1$)
 and $D_{12}$ (the group $G_2$) contains $K_5$. But by Theorem~\ref{1000}, $\mathscr{I}(A_4)$ is planar and by \eqref{e10}, it is acyclic.

 Combining all the cases together, the proof follows.
\end{proof}

\begin{pro}\label{intersecting graph t5}
 If $G$ is a non-abelian group of order $p^ \alpha q$, where $p$, $q$ are distinct primes and $\alpha \geq 3$, then $\gamma(\mathscr{I}(G))>1$,
$\overline{\gamma}(\mathscr{I}(G))>1$ and $\mathscr{I}(G)$ contains $K_5$.
\end{pro}
\begin{proof}
 Let $P$ denote a Sylow p-subgroup of $G$. We shall prove this result by induction on $\alpha$. First we prove this result when
$\alpha =3$. If $p > q$, then $n_p(G)=1$, by Sylow's theorem and our group $G \cong P \rtimes \mathbb Z_q$. Suppose $\gamma(\mathscr{I}(P))>1$,
$\overline{\gamma}(\mathscr{I}(P))>1$ and
$\mathscr{I}(P)$ contains $K_5$,
then the same holds for $\mathscr{I}(G)$ also. So it is enough to consider the cases when $\mathscr{I}(P)$ is one of planar, toroidal, projective-planar,
$K_5$-free, $C_3$-free or bipartite.
By Propositions~\ref{intersecting graph t100}, \ref{intersecting graph t1}, \ref{intersecting graph t2} and Theorem~\ref{1000},
 $P$ is isomorphic to one of $\mathbb Z_{p^3}$, $\mathbb Z_{p^2}\times \mathbb Z_p$ ($p=2$, $3$, $5$) or $M_{p^3}$ ($p=3$, $5$).
\begin{itemize}
\item If $P \cong \mathbb Z_{p^3}$, then $G \cong \mathbb Z_{p^3} \rtimes \mathbb Z_q=\langle a,b~|~a^{p^3}=b^q=1, bab^{-1}, ord_{p^3}(i)=q\rangle$.
Here $H_1:=\langle a \rangle$, $H_2:=\langle a^p, b \rangle$, $H_3:=\langle a^{p^2}, b\rangle$, $H_4:=\langle a^p \rangle$,
$H_5:=\langle a^p, ab\rangle$, $H_6:=\langle a^{p^2}, ab\rangle$, $H_7:=\langle a^p, a^2b\rangle$, $H_8:=\langle a^{p^2} \rangle$ are proper subgroups of $G$.
Further $H_8$ is a subgroup of $H_i$, for every $i=1$, $\ldots$, 7.
It follows that they form $K_8$ as a subgraph of $\mathscr{I}(G)$, so $\gamma(\mathscr{I}(G))>1$ and $\overline{\gamma}(\mathscr{I}(G))>1$.
\item If $P\cong \mathbb Z_{p^2}\times \mathbb Z_p$, $p=2$, $3$, $4$, then $G\cong (\mathbb Z_{p^2}\times \mathbb Z_p)\rtimes \mathbb Z_q$. Here $H_1:=\langle a,c\rangle$, $H_2:=\langle ab,c\rangle$, $H_3:=\langle a^p,b,c\rangle$,
$H_4:=\langle a^p,c\rangle$, $H_5:=\langle a\rangle$, $H_6:=\langle ab\rangle$, $H_7:=\langle a^p,b\rangle$, $H_8:=\langle a^p\rangle$ are subgroups
of $G$, where $\langle a, b\rangle=P$ and $\langle c\rangle=\mathbb Z_q$. Also $H_8$ is a subgroup of $H_i$, $i=1$, $\ldots$, 7. Thus $\mathscr{I}(G)$ has $K_8$
as a subgraph, so $\gamma(\mathscr{I}(G))>1$ and $\overline{\gamma}(\mathscr{I}(G))>1$.
\item If $P\cong M_{p^3}$, $p=3$, $5$,
then $G\cong M_{p^3}\rtimes \mathbb Z_q$. Here  $H_1:=\langle a,c\rangle$, $H_2:=\langle ab,c\rangle$, $H_3:=\langle a^p,b,c\rangle$,
$H_4:=\langle a^2,c\rangle$, $H_5:=\langle a\rangle$, $H_6:=\langle ab\rangle$, $H_7:=\langle a^p,b\rangle$, $H_8:=\langle a^p\rangle$ are subgroups
of $G$, where $\langle a, b\rangle=P$ and $\langle c\rangle=\mathbb Z_q$. Also $H_8$ is a subgroup of $H_i$, $i=1$, $\ldots$, 7. Therefore, $\mathscr{I}(G)$ has $K_8$
as a subgraph and so $\gamma(\mathscr{I}(G))>1$ and $\overline{\gamma}(\mathscr{I}(G))>1$.
\end{itemize}

Now, let us consider the case $p<q$ and $(p,q)\neq (2,3)$. $n_q(G)=p$ is not possible. If $n_q= p^2$, then $q|(p+1)(p-1)$, which implies that $q|(p+1)$ or
$q|(p-1)$. But this only leaves $p^3q-p^3(q-1)=p^3$ elements and our Sylow p-subgroup must be normal, a case we already considered. Therefore,
the only remaining possibility is that $G \cong \mathbb Z_q \rtimes P$. Suppose $\gamma(\mathscr{I}(P))>1$,
$\overline{\gamma}(\mathscr{I}(P))>1$ and
$\mathscr{I}(P)$ contains $K_5$,
then the same holds for $\mathscr{I}(G)$ also. So it is enough to consider the cases when $\mathscr{I}(P)$ is one of planar, toroidal, projective-planar,
$K_5$-free, $C_3$-free and bipartite.
 By Propositions~\ref{intersecting graph t100}, \ref{intersecting graph t1}, \ref{intersecting graph t2}  and Theorem~\ref{1000},
$P$ is isomorphic to one of $\mathbb Z_{p^3}$, $\mathbb Z_{p^2}\times \mathbb Z_p$ ($p=2$, $3$, $5$), $Q_8$, $M_8$ or $M_{p^3}$.
\begin{itemize}
\item If $P\cong \mathbb Z_{p^3}$, then $G \cong \mathbb Z_q \rtimes \mathbb Z_{p^3}=\langle a,b~|~a^q=b^3=1, bab^{-1}=a^i, ord_q(i)=p\rangle$.
Here $H_1:=\langle ab^p\rangle$, $H_2:=\langle b\rangle$, $H_3:=\langle ab\rangle$,
$H_4:=\langle a^2b\rangle$, $H_5:=\langle a^3b\rangle$, $H_6:=\langle ab^{p^2}\rangle$, $H_7:=\langle b^p\rangle$,
$H_8:=\langle b^{p^2}\rangle$ are proper subgroups
of $G$. Also $H_8$ is a subgroup of $H_i$, $i=1$, $\ldots$, $7$. Thus $\mathscr{I}(G)$ has $K_8$
as a subgraph and so $\gamma(\mathscr{I}(G))>1$ and $\overline{\gamma}(\mathscr{I}(G))>1$.
\item If $P\cong \mathbb Z_{p^2}\times \mathbb Z_p$ or $M_{p^3}$, then $G\cong \mathbb Z_q\rtimes (\mathbb Z_{p^2}\times \mathbb Z_p)$ or $\mathbb Z_q\rtimes M_{p^3}$.
Here $H_1:=\langle a,c\rangle$, $H_2:=\langle ab,c\rangle$, $H_3:=\langle a^p,b,c\rangle$,
$H_4:=\langle a^p,c\rangle$, $H_5:=\langle a\rangle$, $H_6:=\langle ab\rangle$, $H_7:=\langle a^p,b\rangle$, $H_8:=\langle a^p\rangle$ are subgroups
of $G$, where $\langle a, b\rangle=P$ and $\langle c\rangle=\mathbb Z_q$. Also $H_8$ is a subgroup of $H_i$, $i=1$, $\ldots$, 7. Thus $\mathscr{I}(G)$ has $K_8$
as a subgraph, so $\gamma(\mathscr{I}(G))>1$ and $\overline{\gamma}(\mathscr{I}(G))>1$.

\item If $P \cong Q_8$, then $G\cong \mathbb Z_q\rtimes Q_8$.  Here $H_1:=\langle a,c\rangle$, $H_2:=\langle b,c\rangle$, $H_3:=\langle ab,c\rangle$,
$H_4:=\langle a^2,c\rangle$, $H_5:=\langle a\rangle$, $H_6:=\langle b\rangle$, $H_7:=\langle ab\rangle$, $H_8:=\langle a^2\rangle$ are proper subgroups
of $G$, where $\langle a, b\rangle=P$ and $\langle c\rangle=\mathbb Z_q$. Also $H_8$ is a subgroup of $H_i$, $i=1$, $\ldots$, 7. Thus these eight subgroups
forms $K_8$
as a subgraph of $\mathscr{I}(G)$ and this implies that $\gamma(\mathscr{I}(G))>1$ and $\overline{\gamma}(\mathscr{I}(G))>1$.
\item If $P \cong M_8$. Here $\mathbb Z_q\rtimes M_8$,
$H_1:=\langle a,c\rangle$, $H_2:=\langle a^2,ab,c\rangle$, $H_3:=\langle a^2,b,c\rangle$,
$H_4:=\langle a^2,c\rangle$, $H_5:=\langle a\rangle$, $H_6:=\langle a^2,b\rangle$, $H_7:=\langle a^2,ab\rangle$, $H_8:=\langle a^2\rangle$ are proper subgroups
of $G$, where $\langle a, b\rangle=P$ and $\langle c\rangle=\mathbb Z_q$.
Also $H_8$ is a subgroup of $H_i$, $i=1$ to 7. Thus these eight subgroups form $K_8$
as a subgraph of $\mathscr{I}(G)$ and so $\gamma(\mathscr{I}(G))>1$, $\overline{\gamma}(\mathscr{I}(G))>1$.
\end{itemize}
If $(p,q)= (2, 3)$,
then $G \cong S_4$. In this case, $G$ has at least two copies of $D_8$. So by Figure~\ref{int f1}(c), $D_8$ together with its proper subgroups form $K_5$ as a subgraph of $\mathscr {I}(G)$. Then $\mathscr{I}(G)$ has a subgraph $\mathcal B_1$ as shown in Figure~\ref{fig:int f3}, so $\gamma(\mathscr{I}(G))>1$ and $\overline{\gamma}(\mathscr{I}(G))>1$.
Thus from the above arguments the result is true when $\alpha =3$.

Assume that the result is true for all non-abelian group of order $p^mq$, with $m< \alpha$. We prove the result when $\alpha >3$. If $n_p(G)=1$,
then $G\cong P \rtimes \mathbb Z_q$. Suppose $\gamma(\mathscr{I}(P))>1$,
$\overline{\gamma}(\mathscr{I}(P))>1$ and
$\mathscr{I}(P)$ contains $K_5$,
then the same holds for $\mathscr{I}(G)$ also. So it is enough to consider the cases when $\mathscr{I}(P)$ is one of planar, toroidal, projective-planar,
$K_5$-free, $C_3$-free and bipartite.
By Propositions~\ref{intersecting graph t100}, \ref{intersecting graph t1},
\ref{intersecting graph t2} and Theorem~\ref{1000}, $P \cong
\mathbb Z_{p^ \alpha}$ or $M_{2^4}$.
\begin{itemize}
\item If $P\cong \mathbb Z_{p^\alpha}$, then $G$ has a subgroup
$\langle a^p, b \rangle \cong \mathbb Z_{p^3} \rtimes \mathbb Z_q$. So by induction hypothesis
$\gamma(\mathscr{I}( \langle a^p, b \rangle))>1$, $\overline{\gamma}(\mathscr{I}( \langle a^p, b \rangle))>1$, and so $\gamma(\mathscr{I}(G))>1$ and $\overline{\gamma}(\mathscr{I}(G))>1$.
\item  If $P\cong M_{2^4}$, then by \eqref{e101},
$P$ together with its proper subgroups form $K_8$ as a subgraph of $\mathscr{I}(G)$ and so $\overline{\gamma}(\mathscr{I}(G))>1$ and $\gamma(\mathscr{I}(G))>1$.
\end{itemize}
Let $n_p(G) \neq 1$. Since $G$ is solvable, $G$ has a normal subgroup
$N$ of order $p^{\alpha -1}q$.
Suppose $\gamma(\mathscr{I}(N))>1$,
$\overline{\gamma}(\mathscr{I}(N))>1$ and
$\mathscr{I}(N)$ contains $K_5$,
then the same holds for $\mathscr{I}(G)$ also. So it is enough to consider the cases when $\mathscr{I}(N)$ is one of planar, toroidal, projective-planar,
$K_5$-free, $C_3$-free and bipartite.
By Propositions~\ref{intersecting graph t100}, \ref{intersecting graph t1},
\ref{intersecting graph t4} and Theorem~\ref{1000},
 $N\cong \mathbb Z_{p^3q}$.

 Suppose $G$ has two elements, say $a$, $b$ of orders $p$, $q$ respectively with $b\in N$, $a\notin N$, then
 $H:=\langle a,b\rangle$ is a proper subgroup of $G$. Consider $N$ together with its proper subgroups
$H_1$, $H_2$, $H_3$, $H_4$, $H_5$, $H_6$, $H_7$, of order $p$, $p^2$, $p^3$, $pq$, $p^2q$, $p^3q$, $q$ respectively.
Here $H_i\cap H_j\neq \{e\}$, for every $i$, $j=1$ to $6$; $H_7$ is a subgroup of $H_i$, $i=4$, 5, 6,
$H$; $H\cap H_i=H_7$, $i=4$, 5, 6. It follows that $\mathscr{I}(G)$ has a subgraph $\mathcal A_1$ as shown in Figure~\ref{fig:f2},
so $\gamma(\mathscr{I}(G))>1$ and $\overline{\gamma}(\mathscr{I}(G))>1$.

Suppose every subgroup of $G$ of order $p$ is  contained in $N$, then $G$ has at least three Sylow $p$-subgroups, let them be $P_1$, $P_2$, $P_3$. Consider $N$ together with its proper subgroups
$H_1$, $H_2$, $H_3$, $H_4$, $H_5$, $H_6$ of order $p$, $p^2$, $p^3$, $pq$, $p^2q$, $p^3q$, respectively. Here $H_1$ is a subgroup of $H_i$, $i=2$, $\ldots$, 6, $N$, $P_1$, so they form $K_8$ as a subgraph of $\mathscr{I}(G)$. Therefore, $\gamma(\mathscr{I}(G))>1$ and $\overline{\gamma}(\mathscr{I}(G))>1$.

Thus the result is true when $\alpha>3$, it follows that the result is true for all $\alpha\geq 3$.
\end{proof}

\begin{pro}\label{intersecting graph t6}
 If $G$ is a non-abelian group of order $p^2q^2$, where $p$, $q$ are distinct primes, then $\gamma(\mathscr{I}(G))>1$,  $\overline{\gamma}(\mathscr{I}(G))>1$ and $\mathscr{I}(G)$ contains $C_3$.  Also  $\mathscr{I}(G)$ is $K_5$-free if and only if $G \cong \langle a,b,c~|~a^p=b^p=c^{q^2}=1, ab=ba, cac^{-1}=ab^{-1}, cbc^{-1}=ab^l\rangle$,
 where $\bigl(\begin{smallmatrix}
  0 & -1\\ 1 & l
\end{smallmatrix} \bigr)$ has order $q^2$ in $GL_2(p)$, $q^2~|~(p+1)$.

\end{pro}
\begin{proof}
 We use the classification of groups of order $p^2q^2$ given in \cite{lin}.

 Let $P$ and $Q$ denote a Sylow $p$, $q$-subgroups of $G$ respectively,
 with out loss of generality, we assume that $p > q$. By Sylow's theorem,  $n_p(G)=1, q, q^2$. But $n_p(G)= q$ is not possible,
since $p> q$. If $n_p(G)= q^2$, then $p~|~(q+1)(q-1)$, this implies that $p~|~(q+1)$, which is true only
 when $(p, q)= (3, 2)$.

\noindent \textbf{Case 1:} $(p, q) \neq (3, 2)$. Then $G \cong P \rtimes Q$.

\noindent\textbf{Subcase 1a:} If $G \cong \mathbb Z_{p^2} \rtimes \mathbb Z_{q^2}= \langle a, b| a^{p^2}=b^{q^2}=1, bab^{-1}=a^i, i^{q^2} \equiv 1 (\mbox{mod}~ p^2) \rangle$, then we have
$H:=\langle a^p, b\rangle\cong \mathbb Z_p\rtimes \mathbb Z_{q^2}$, so by \eqref{e1}, $H$ together with its proper subgroups forms $K_8$ as a subgraph of
$\mathscr{I}(G)$. Hence $\gamma(\mathscr{I}(G))>1$ and $\overline{\gamma}(\mathscr{I}(G))>1$.

\noindent \textbf{Subcase 1b:} If $G \cong \mathbb Z_{p^2} \rtimes (\mathbb Z_q \times \mathbb Z_q)$, then $H_1:=\langle a \rangle=\mathbb Z_{p^2}$, $H_2:=\langle a^p \rangle$,
$H_3:=\langle a, c \rangle$, $H_4:=\langle a^p, b \rangle$, $H_5:=\langle a, b \rangle$, $H_6:=\langle a^p, c\rangle$,
$H_7:=\langle b,c\rangle=\mathbb Z_q\times \mathbb Z_q$,
$H_8:=\langle a^p,b,c\rangle$ are proper subgroups of $G$, where $\langle a\rangle=\mathbb Z_{p^2}$ and $\langle b,c\rangle=\mathbb Z_q\times \mathbb Z_q$.
Here $H_i$, $i=3$, $\ldots$, $8$ intersect with each other nontrivially;
$H_2$ is a subgroup of $H_1$, $H_3$, $H_4$;
$H_1\cap H_4=H_2$; $H_1$ is a subgroup of $H_3$.
 It follows that $\mathscr{I}(G)$ has a subgraph $\mathcal A_1$ as shown in Figure~\ref{fig:f2}, so $\gamma(\mathscr{I}(G))>1$ and $\overline{\gamma}(\Gamma_(G))>1$.

\noindent \textbf{Subcase 1c:} If $G \cong (\mathbb Z_p \times \mathbb Z_p) \rtimes \mathbb Z_{q^2}:=\langle a,b,c~|~a^p=b^p=c^{q^2}=1, ab=ba, cac^{-1}=a^ib^j, cbc^{-1}=a^kb^l\rangle$,
 where $\bigl(\begin{smallmatrix}
  i & j\\ k & l
\end{smallmatrix} \bigr)$ has order $q^2$ in $GL_2(p)$, $q^2~|~(p+1)$.
Then we have two possibilities:
\begin{itemize}
\item Suppose $G$ has a subgroups of order $pq^2$ and $pq$, then $H_1:=\langle c \rangle=\mathbb Z_{q^2}$, $H_2:=\langle c^p \rangle$,
$H_3:=\langle a, c \rangle$, $H_4:=\langle b, c^p \rangle$, $H_5:=\langle b, c\rangle$, $H_6:=\langle a,c^p\rangle$,
$H_7:=\langle a,b\rangle=\mathbb Z_p\times \mathbb Z_p$,
$H_8:=\langle a,b,c^p\rangle$ are proper subgroups of $G$. Here $H_i$, $i=3$, $\ldots$, $8$ intersect with each other nontrivially;
$H_2$ is a subgroup of $H_1$, $H_3$, $H_4$;
$H_1$ is a subgroup of $H_3$; $H_1\cap H_4=H_2$.
 It follows that $\mathscr{I}(G)$ has a subgraph $\mathcal A_1$ as shown in Figure~\ref{fig:f2}, so $\gamma(\mathscr{I}(G))>1$, $\overline{\gamma}(\mathscr{I}(G))>1$.
\item Suppose $G$ has a no subgroups of order $pq^2$ and $pq$, then by Theorem~\ref{1000}, $\mathscr{I}(G)$ is planar and by \eqref{33}, $\mathscr{I}(G)$ contains $C_3$.
\end{itemize}

\noindent \textbf{Subcase 1d:} If $G \cong (\mathbb Z_p \times \mathbb Z_p) \rtimes (\mathbb Z_q \times \mathbb Z_q)$, then $H_1:=\langle a,b \rangle=\mathbb Z_p\times \mathbb Z_p$,
$H_2:=\langle a,c\rangle$,
$H_3:=\langle a,d\rangle$, $H_4:=\langle b,c \rangle$, $H_5:=\langle a,b,c \rangle$, $H_6:=\langle a,b,d\rangle$, $H_7:=\langle a,c,d\rangle$,
$H_8:=\langle b,c,d\rangle$ are proper subgroups of $G$, where $\langle a,b\rangle:=\mathbb Z_p\times \mathbb Z_p$ and
$\langle c,d\rangle=\mathbb Z_q\times \mathbb Z_q$.
Here $H_i$, $i=1$, $\ldots$, $8$ intersect with each other nontrivially.
 It follows that $\mathscr{I}(G)$ contains $K_8$ as a subgraph and hence $\gamma(\mathscr{I}(G))>1$ and $\overline{\gamma}(\mathscr{I}(G))>1$.

\noindent \textbf{Case 2:} $(p, q)= (3, 2)$. Up to isomorphism, there are nine groups of order 36. In the following we consider each of these groups.
\begin{enumerate}[{\normalfont (i)}]
\item $G\cong D_{18}$. Here $H_1:=\langle a\rangle$, $H_2:=\langle a^2\rangle$,
$H_3:=\langle a^6\rangle$, $H_4:=\langle a^6,b\rangle$, $H_5:=\langle a^6, ba\rangle$, $H_6:=\langle a^6, ba^2\rangle$, $H_7:=\langle a^6,ba^3\rangle$,
$H_8:=\langle a^6, ba^4\rangle$ are proper subgroups of $G$. Also $H_3$ is a subgroup of $H_i$, $i=1$, $\ldots$, $8$.
 It follows that $\mathscr{I}(G)$ contains $K_8$ as a subgraph, so $\gamma(\mathscr{I}(G))>1$ and $\overline{\gamma}(\mathscr{I}(G))>1$.

\item $G\cong S_3\times S_3$. Here $H_1:=S_3\times \{e\}$, $H_2:=S_3\times \langle (123)\rangle$,
$H_3:=S_3\times \langle (12)\rangle$, $H_4:=S_3\times \langle (13)\rangle$, $H_5:=S_3\times \langle (23)\rangle$,
$H_6:=\langle (13)\rangle\times S_3$, $H_7:=\langle (123)\rangle\times S_3$,
$H_8:=\langle (12)\rangle\times S_3$ are proper subgroups of $G$.
Also $H_i$, $i=1$, $\ldots$, $8$ intersect with each other nontrivially, it follows that $\mathscr{I}(G)$ contains $K_8$ as a subgraph and so
$\gamma(\mathscr{I}(G))>1$ and $\overline{\gamma}(\mathscr{I}(G))>1$.

\item $G\cong \mathbb Z_3\times A_4$. Here $H_1:=\mathbb Z_3\times \{e\}$, $H_2:=\mathbb Z_3\times \langle (123)\rangle$,
$H_3:=\mathbb Z_3\times \langle (124)\rangle$, $H_4:=\mathbb Z_3\times \langle (134)\rangle$, $H_5:=\mathbb Z_3\times \langle (234)\rangle$,
$H_6:=\mathbb Z_3\times \langle (12)(34)\rangle$, $H_7:=\mathbb Z_3\times \langle (13)(24)\rangle$,
$H_8:=\mathbb Z_3\times \langle (14)(23)\rangle$ are proper subgroups of $G$.
Also $H_1$ is a subgroup of $H_i$, $i=2$, $\ldots$, $8$, it follows that $\mathscr{I}(G)$ contains $K_8$ as a subgraph and so $\gamma(\mathscr{I}(G))>1$
and $\overline{\gamma}(\mathscr{I}(G))>1$.

\item $G\cong \mathbb Z_6\times S_3$. Here $H_1:=\mathbb Z_6\times \langle (123)\rangle$, $H_2:=\mathbb Z_6\times \langle (12)\rangle$,
$H_3:=\mathbb Z_6\times \langle (13)\rangle$, $H_4:=\mathbb Z_6\times \langle (23)\rangle$, $H_5:=\mathbb Z_3\times \langle (123)\rangle$,
$H_6:=\mathbb Z_3\times \langle (12)\rangle$, $H_7:=\mathbb Z_3\times \langle (13)\rangle$,
$H_8:=\mathbb Z_3\times \langle (23)\rangle$ are proper subgroups of $G$.
Also $H_i$, $i=1$, $\ldots$, $8$ intersect with each other nontrivially, it follows that $\mathscr{I}(G)$ contains $K_8$ as a subgraph and hence
$\gamma(\mathscr{I}(G))>1$ and $\overline{\gamma}(\mathscr{I}(G))>1$.

\item $G\cong \mathbb Z_9\rtimes\mathbb Z_4=\langle a,b~|~a^9=b^4=1, bab^{-1}=a^i, i^4\equiv 1(\mbox{mod}~ 9)\rangle$. Here $H_1:=\langle a\rangle$,
$H_2:=\langle a, b^2\rangle$,
$H_3:=\langle a^3\rangle$, $H_4:=\langle a^3,b\rangle$, $H_5:=\langle a^3, b^2\rangle$, $H_6:=\langle a^3, ab^2\rangle$, $H_7:=\langle a^3, a^2b^2\rangle$,
$H_8:=\langle b\rangle$ are proper subgroups of $G$. Also $H_3$ is a subgroup of $H_i$, $i=1$, $\ldots$, 7; $H_8$ intersect with $H_2$, $H_4$, $H_5$ non-trivially. It follows that $\mathscr{I}(G)$ has a subgraph $\mathcal A_1$ as shown in Figure~\ref{fig:f2}, so $\gamma(\mathscr{I}(G))>1$, $\overline{\gamma}(\mathscr{I}(G))>1$.

\item $G\cong \mathbb Z_3\times(\mathbb Z_3\rtimes \mathbb Z_4)=\langle a,b,c~|~a^3=b^3=c^4=1, ab=ba, ac=ca, cbc^{-1}=b^i,
\mbox{ord}_2(i)=3\rangle$. Here $H:=\mathbb Z_3\rtimes \mathbb Z_4$
 is a subgroup of $G$, and as in Case 1b, in the proof of Theorem~\ref{intersecting graph t4}, $H$ has a unique subgroup of order $2$, let it be $H_1$;
$H_i$, $i=2$, 3, 4 be subgroups of $H$ of order $4$; $H_5$ be a subgroup of $H$ of order $6$; it follows that $H$ together with its proper subgroups forms $K_6$
as a subgraph of $\mathscr{I}(H)$.
Moreover $H_6:=\mathbb Z_3\times H_1$ and
$H_7:=\mathbb Z_3\times H_2$ are subgroups of $G$. Thus $H_1$ is a subgroup of $H$, $H_i$, $i=1$, 2, 3, 4, 6, 7. It follows that $G$
contains $K_8$ as a subgraph, and so $\gamma(\mathscr{I}(G))>1$ and $\overline{\gamma}(\mathscr{I}(G))>1$.

\item $G\cong (\mathbb Z_3\times \mathbb Z_3)\rtimes \mathbb Z_4:=\langle a,b,c~|~a^3=b^3=c^4=1, ab=ba, cac^{-1}=a^ib^j, cbc^{-1}=a^kb^l\rangle$,
 where $\bigl(\begin{smallmatrix}
  i & j\\ k & l
\end{smallmatrix} \bigr)$ has order $4$ in $GL_2(3)$. We already discussed the intersection graph of subgroups of this group in Subcase 1c.

\item $G\cong \mathbb Z_2\times (\mathbb Z_3\times \mathbb Z_3)\rtimes \mathbb Z_2$. Here $H_1:=\langle a,b, c\rangle$,
$H_2:=\langle a, b,d\rangle$,
$H_3:=\langle a,c,d\rangle$, $H_4:=\langle b,c,d\rangle$, $H_5:=\langle a, b\rangle$, $H_6:=\langle a, c\rangle$, $H_7:=\langle a,d\rangle$,
$H_8:=\langle b,c\rangle$ are proper subgroups of $G$. Also these subgroups intersect with each other non-trivially.
 It follows that $\mathscr{I}(G)$ contains $K_8$ as a subgraph, so $\gamma(\mathscr{I}(G))>1$ and $\overline{\gamma}(\mathscr{I}(G))>1$.

\item $G\cong (\mathbb Z_2\times \mathbb Z_2)\rtimes \mathbb Z_9$. Here $H_1:=\langle a,b \rangle$, $H_2:=\langle a,c\rangle$,
$H_3:=\langle a, c^3 \rangle$, $H_4:=\langle b,c\rangle$, $H_5:=\langle b,c^3\rangle$, $H_6:=\langle a,b,c^3\rangle$, $H_7:=\langle c\rangle$,
$H_8:=\langle c^3\rangle$ are proper subgroups of $G$. Also $H_i$, $i=1$, $\ldots$, $6$ intersect with each other nontrivially;
$H_7$ is a subgroup of $H_8$, $H_3$, $H_4$;
$H_1$ intersect with $H_3$, $H_4$ non-trivially.
 It follows that $\mathscr{I}(G)$ has a subgraph $\mathcal A_1$ as shown in Figure~\ref{fig:f2},
so $\gamma(\mathscr{I}(G))>1$ and $\overline{\gamma}(\mathscr{I}(G))>1$.
\end{enumerate}

The proof follows by combining all the above cases.
\end{proof}
\begin{pro}\label{intersecting graph t7}
 If $G$ is a non-abelian group of order $p^{\alpha}q^{\beta}$, where $p$, $q$ are distinct primes and $\alpha$, $\beta \geq 2$,
then $\gamma(\mathscr{I}(G))>1$, $\overline{\gamma}(\mathscr{I}(G))>1$ and $\mathscr{I}(G)$ contains $K_5$.
\end{pro}
\begin{proof}
We prove the result by induction on
$\alpha + \beta$. If $\alpha+\beta=5$, then $|G|=p^3q^2$. Since $G$ is solvable,
it has a normal subgroup $N$ of prime index.

\noindent\textbf{Case 1:} If $[G:N]=p$, then $|N|=p^2q^2$. Suppose $\gamma(\mathscr{I}(N))>1$,
$\overline{\gamma}(\mathscr{I}(N))>1$ and
$\mathscr{I}(N)$ contains $K_5$,
then the same holds for $\mathscr{I}(G)$ also. So it is enough to consider the cases when $\mathscr{I}(N)$ is one of planar, toroidal, projective-planar,
$K_5$-free, $C_3$-free and bipartite.
By Propositions~\ref{intersecting graph t100}, \ref{intersecting graph t1},  \ref{intersecting graph t6} and Theorem~\ref{1000}, $N\cong \mathbb Z_{p^2q^2}$ or $\langle a,b,c~|~a^p=b^p=c^{q^2}=1, ab=ba, cac^{-1}=b^{-1}, cbc^{-1}=a^1b^l\rangle$,
 where $\bigl(\begin{smallmatrix}
  0 & -1\\ 1 & l
\end{smallmatrix} \bigr)$ has order $q^2$ in $GL_2(p)$, $q^2~|~(p+1)$.
 If $N\cong \mathbb Z_{p^2q^2}$, then $N$ together with its proper subgroups forms a  subgraph in $\mathscr{I}(G)$, which is isomorphic to $\mathcal A_1$ as shown in Figure~\ref{fig:f2} and so $\gamma(\mathscr{I}(G))>1$, $\overline{\gamma}(\mathscr{I}(G))>1$.

If $N$ is isomorphic to the second group, then $N$ has unique subgroups of order $p^2$ and $p^2q$, let them be $H_1$, $H_2$; also there are $p+1$ subgroups of order $p$, let them be $B_i$, $i=1$, 2, $\ldots$, $p+1$. Here $p>q$. Suppose $p\geq 5$, then $\mathscr{I}(G)$ contains $K_{4,5}$ as a subgraph with bipartition $X:=\{P$, $H_1$, $H_2$, $N\}$ and $Y:=\{B_1$, $B_2$, $B_3$, $B_4$, $B_5\}$, where $P$ is a Sylow $p$-subgroup of $G$ containing $H_1$ and $H_1\cong \mathbb Z_p\times \mathbb Z_p$. Here $B_1$ is a subgroup of $N$, $H_1$, $H_2$, $P$ and so they form $K_5$ as a subgraph of $\mathscr{I}(G)$. If $p=3$, then $P\cong \mathbb Z_{p^2}\times \mathbb Z_p$, $M_{p^3}$ or $(\mathbb Z_p\times \mathbb Z_p)\rtimes \mathbb Z_p$. If $P\cong \mathbb Z_{p^2}\times \mathbb Z_p$ or $M_{p^3}$, then by \eqref{e5}, \eqref{e100}, $H_1$, $H_2$ together with the proper subgroups of $H_1$ forms a subgraph of $\gamma(\mathscr{I}(G))$, which is  isomorphic to $\mathcal A_1$ shown in Figure~\ref{fig:f2}, so $\gamma(\mathscr{I}(G))>1$ and $\overline{\gamma}(\mathscr{I}(G))>1$. If $P\cong (\mathbb Z_p\times \mathbb Z_p)\rtimes \mathbb Z_p$, then by Proposition~\ref{intersecting graph t2}, $\gamma(\mathscr{I}(G))>1$ and $\overline{\gamma}(\mathscr{I}(G))>1$.

\noindent\textbf{Case 2:} If $[G:N]=q$, then $|N|=p^3q$. Suppose $\gamma(\mathscr{I}(N))>1$,
$\overline{\gamma}(\mathscr{I}(N))>1$ and
$\mathscr{I}(N)$ contains $K_5$,
then the same holds for $\mathscr{I}(G)$ also. So it is enough to consider the cases when $\mathscr{I}(P)$ is one of planar, toroidal, projective-planar,
$K_5$-free, $C_3$-free and bipartite. By Propositions~\ref{intersecting graph t100}, \ref{intersecting graph t1},
\ref{intersecting graph t5} and Theorem~\ref{1000}, $N\cong \mathbb Z_{p^3q}$. Let $H_i$, $i=1$, 2, 3, 4, 5, 6 be subgroups of $N$ of order $p$, $p^2$, $p^3$, $pq$, $p^2q$, $q$ respectively.  Let $H$ be a subgroup of $G$ of order $q^2$ such that $H$ contains $H_6$. By \eqref{201}, $N$ together with its proper subgroups and $H$ forms a subgraph in $\mathscr{I}(G)$, which is isomorphic to $\mathcal A_1$ as shown in Figure~\ref{fig:f2}, so $\gamma(\mathscr{I}(G))>1$ and $\overline{\gamma}(\mathscr{I}(G))>1$.

Now we assume that the
result is true for all non-abelian groups of order $p^mq^n$, where $m+n < \alpha + \beta$ ($m+n\geq 5$, $m,n\geq 2$). We prove the result when $\alpha + \beta >5$.
Since $G$ is solvable, $G$ has a subgroup $H$ of prime index, with out loss of generality, say $q$, and so $|H|=p^{\alpha}q^{\beta -1}$. If $H$ is
cyclic, then $H$ together with its proper subgroups forms $K_6$ as a subgraph of $\mathscr{I}(G)$. Now let $K$ be the subgroup of $H$ of order $q^{\beta-1}$
and let $Q$ be a $q$-Sylow subgroup of $G$ containing $K$. Also let $H_i$, $i=1$, 2, 3 be subgroups of $H$ of order $pq$, $p^2q$, $p^3q$ respectively.
It follows that $\mathscr{I}(G)$ has a subgraph $\mathcal A_1$ as shown in Figure~\ref{fig:f2},  so $\gamma(\mathscr{I}(G))>1$ and $\overline{\gamma}(\mathscr{I}(G))>1$.

If $H$ is non-cyclic
abelian, then by Proposition~\ref{intersecting graph t1}, $\gamma(\mathscr{I}(H))>1$, $\overline{\gamma}(\mathscr{I}(H))>1$
and $\mathscr{I}(H)$ contains $K_5$ as a subgraph, so $\gamma(\mathscr{I}(G))>1$, $\overline{\gamma}(\mathscr{I}(G))>1$
and $\mathscr{I}(G)$ contains $K_5$.
If $H$ is non-abelian, then we have the following cases to consider:

\noindent\textbf{Case a:} If $\beta = 2$, then $\alpha > 3$. So by Proposition~\ref{intersecting graph t2}, $\gamma(\mathscr{I}(H))>1$, $\overline{\gamma}(\mathscr{I}(H))>1$
 and $\mathscr{I}(H)$ contains $K_5$ as a subgraph.

\noindent\textbf{Case b:} If $\beta > 2$, then by induction hypothesis, $\gamma(\mathscr{I}(H))>1$, $\overline{\gamma}(\mathscr{I}(H))>1$
 and $\mathscr{I}(H)$ contains $K_5$ as a subgraph.

\noindent\textbf{Case c:} If $\alpha = 2$, then $\beta > 2$. By Case b, $\gamma(\mathscr{I}(H))>1$, $\overline{\gamma}(\mathscr{I}(H))>1$
 and $\mathscr{I}(H)$ contains $K_5$ as a subgraph.

\noindent\textbf{Case d:} If $\alpha >2$, then by induction hypothesis, $\gamma(\mathscr{I}(H))>1$, $\overline{\gamma}(\mathscr{I}(H))>1$
 and $\mathscr{I}(H)$ contains $K_5$ as a subgraph.

It follows that $\gamma(\mathscr{I}(G))>1$, $\overline{\gamma}(\mathscr{I}(G))>1$ and $\mathscr{I}(G)$ contains $K_5$ as a subgraph.
\end{proof}

\begin{pro}\label{intersecting graph t8}
 Let $G$ be a non-abelian solvable group of order $pqr$, where $p$, $q$, $r$ are distinct primes, $p>q>r$. Then
 \begin{enumerate}[\normalfont (1)]
 \item $\mathscr{I}(G)$ contains $C_3$;
\item $\mathscr{I}(G)$ is $K_5$-free if and only if $G \cong \langle a,b,c~|~a^p=b^q=c^r=1, b^{-1}ab=a^\mu, c^{-1}ac=a^v, bc=cb\rangle$, where $r$, $q$ are divisor of $(p-1)$ and $v$, $\mu\neq 1$;
\item $\mathscr{I}(G)$ is non-toroidal and non-projective-planar;
\end{enumerate}
\end{pro}
\begin{proof}
\noindent By~\cite[p. 215]{cole},
up to isomorphism there are four groups of order $pqr$. In the following we deal with each of these groups.

\noindent\textbf{Case a:} If $r\nmid (p-1)$ and $q\nmid (p-1)$, then $G\cong \mathbb Z_{pqr}$, which is not possible.

\noindent\textbf{Case b:} If $r |(p-1)$, $r \nmid (q-1)$ and $q\nmid (p-1)$, then $G\cong\langle a,b,c~|~a^p=b^q=c^r=1, ab=ba, ac=ca, c^{-1}bc=b^v\rangle$, $v\neq 1$.
Here $H_1:=\langle a\rangle$, $H_2:=\langle b\rangle$, $H_3:=\langle a,b\rangle$, $H_4:=\langle b,c\rangle$,
$I_1:=\langle c\rangle$, $I_2:=\langle bc\rangle$, $I_3:=\langle b^2c\rangle$,$\ldots$, $I_q:=\langle b^{q-1}c\rangle$, $B_1:=\langle a,c\rangle$,
 $B_2:=\langle a,bc\rangle$, $\ldots$, $B_q:=\langle a,b^{q-1}c\rangle$ are the only subgroups of $G$.
If $q=3$, then $r=2$, and so $r$ divides $(q-1)$,  which is not possible.

If $q\geq 5$, then $B_i$, $i=1 \ldots , 5$, $H_1$ intersect with each other non-trivially; $H_2$ is a subgroup of $H_3$, $H_4$;
$H_3\cap B_i=H_1$. It follows that $\mathscr{I}(G)$ has a subgraph $\mathcal A_1$ as shown in Figure~\ref{fig:f2}, so
$\gamma(\mathscr{I}(G))>1$ and $\overline{\gamma}(\mathscr{I}(G))>1$.

\noindent\textbf{Case c:} If $r |(p-1)$, $r \nmid (q-1)$ and $q\nmid (p-1)$, then the group is different from the group given in Case b only in the exchange the roles of $a$, $b$, so $G\cong\langle a,b,c~|~a^p=b^q=c^r=1, ab=ba, bc=cb, c^{-1}ac=a^\mu\rangle$, $\mu\neq 1$.
Here $H_1:=\langle a\rangle$, $H_2:=\langle b\rangle$, $H_3:=\langle a,b\rangle$, $H_4:=\langle a,c\rangle$,
$I_1:=\langle c\rangle$, $I_2:=\langle ac\rangle$, $I_3:=\langle a^2c\rangle$,$\ldots$, $I_q:=\langle a^{q-1}c\rangle$, $B_1:=\langle a,c\rangle$,
$B_2:=\langle b,ac\rangle$, $\ldots$, $B_q:=\langle b,a^{q-1}c\rangle$ are the only subgroups of $G$.
Here also $q=3$ not possible.
If $q\geq 5$, then $B_i$, $i=1$ to 5, $H_1$ intersect with each other non-trivially; $H_2$ is a subgroup of $H_3$, $H_4$;
$H_3\cap B_i=H_1$. It follows that $\mathscr{I}(G)$ has a subgraph $\mathcal A_1$ as in Figure~\ref{fig:f2}, so $\gamma(\mathscr{I}(G))>1$ and $\overline{\gamma}(\mathscr{I}(G))>1$.

\noindent\textbf{Case d:} If $r|(p-1)$, $r | (q-1)$ and  $q\nmid (p-1)$, then
$G\cong \langle a,b,c~|~a^p=b^q=c^r=1, ab=ba, c^{-1}ac=a^\mu, c^{-1}bc=b^v\rangle$, $v$, $\mu\neq 1$.
Here $H_1:=\langle a\rangle$, $H_2:=\langle a,b\rangle$,
$H_3:=\langle a,c\rangle$, $H_4:=\langle a,bc\rangle$, $H_5:=\langle a, b^2c\rangle$, $H_6:=\langle b,c\rangle$, $H_7:=\langle b,ac\rangle$,
$H_8:=\langle b,a^2c\rangle$, $H_9:=\langle b,a^4c\rangle$, $H_{10}:=\langle b,a^{p-1}c\rangle$ are proper subgroups of $G$.
Also $H_1$ is a subgroup of $H_i$, $i=2$, 3, 4, 5
and so they form $K_5$ as a subgraph of $\mathscr{I}(G)$. $\langle b\rangle$ is a subgroup of $H_i$, $i=6$, 7, 8, 9, 10, so they form $K_5$ as a subgraph
of $\mathscr{I}(G)$. It follows that $\mathscr{I}(G)$ has a subgraph $\mathcal B_1$ as in Figure~\ref{fig:int f3}, so $\gamma(\mathscr{I}(G))>1$ and
$\overline{\gamma}(\mathscr{I}(G))>1$.

\noindent\textbf{Case e:} If $q|(p-1)$, then we have the group is of essentially the same form as Case 1b and Case 1c. The elements c, a, b
here playing the same role as a, b, c in Case 1b. The toroidality and the projective-planarity of the intersection graphs of subgroups of this group are as described in Case 1b and Case 1c.

\noindent\textbf{Case f:} If $r |(p-1)$ and $q|(p-1)$, then
$G\cong \langle a,b,c~|~a^p=b^q=c^r=1, b^{-1}ab=a^\mu, c^{-1}ac=a^v, bc=cb\rangle$, $v$, $\mu\neq 1$. By Theorem~\ref{1000}, $\mathscr{I}(G)$ is planar and by Figure~\ref{int f1}(d), it contains $C_3$.

Combining together all the cases, the proof follows.
\end{proof}

\begin{pro}\label{intersecting graph t10}
 Let $G$ be a non-abelian solvable group of order $p^2qr$, where $p$, $q$, $r$ are distinct primes. Then
 \begin{enumerate}[\normalfont (1)]
 \item $\mathscr{I}(G)$ contains $C_3$;
\item $\mathscr{I}(G)$ is $K_5$-free if and only if $G \cong \langle a, b, c~|~a^p=b^p=c^{qr}=1, ab=ba, cac^{-1}=b, cbc^{-1}=ab^l\rangle$, where $l$ is any integer with $\bigl(\begin{smallmatrix}
   0 & -1\\ 1 & l
 \end{smallmatrix} \bigr)$ has order $qr$ in $GL_2(p)$, $qr|(p+1)$.

 \item $\mathscr{I}(G)$ is toroidal if and only if $G \cong \langle a, b, c~|~a^5=b^5=c^6=1, ab=ba, cac^{-1}=b, cbc^{-1}=ab^l\rangle$, where $l$ is any integer with $\bigl(\begin{smallmatrix}
   0 & -1\\ 1 & l
 \end{smallmatrix} \bigr)$ has order $6$ in $GL_2(5)$.
 \item $\mathscr{I}(G)$ is non projective-planar.
\end{enumerate}
\end{pro}
\begin{proof}

\noindent Since $G$ is solvable and it has a Sylow basis $\{P,Q,R\}$, where $P$, $Q$, $R$ denotes the Sylow p,q,r-subgroups of $G$ respectively. Then $PQ$, $PR$, $QR$ are subgroups of $G$ of order $p^2q$, $p^2r$, $qr$ respectively.

\noindent \textbf{Case 1:} A subgroup of $G$ of index $r$ is normal in $G$:
With out loss of generality, we may assume that $PQ$ is normal in $G$.

\noindent \textbf{Subcase 1a:} Suppose $P$ is normal in $PQ$,  then by Theorem~\ref{501}, $P$ is normal in $G$ . Then two possibilities arise:
\begin{enumerate}[\normalfont (i)]
\item  $P\cong \mathbb Z_{p^2}$: Let $H_1$ be the subgroup of $P$ of order $p$. Then $H_1$ is also normal in $G$. Here $H_1R$, $H_1Q$, $H_1QR$ are subgroups of $G$. Also $H_1$ is a subgroup of $H_1Q$, $H_1R$, $P$, $PQ$, $PR$; $QR$ intersect with $PQ$, $PR$, $H_1Q$, $H_1QR$ non-trivially. It follows that $\mathscr{I}(G)$ has a subgraph isomorphic to $\mathcal A_1$ as shown in Figure~\ref{fig:f2}, so $\gamma(\mathscr{I}(G))>1$ and $\overline{\gamma}(\Gamma(G))>1$.

\item  $P\cong \mathbb Z_p\times \mathbb Z_p$: Then $P$ has atleast three subgroups of order $p$, let them be $H_i$, $i= 1, 2, 3$. Suppose $\gamma(\mathscr{I}(PQ))>1$,
$\overline{\gamma}(\mathscr{I}(PQ))>1$ and
$\mathscr{I}(PQ)$ contains $K_5$,
then the same holds for $\mathscr{I}(G)$ also. So it is enough to consider the cases when $\mathscr{I}(PQ)$ is one of planar, toroidal, projective-planar,
$K_5$-free, $C_3$-free and bipartite. By Theorem~\ref{1000}, Propositions~\ref{intersecting graph t1} and \ref{intersecting graph t4}, the only possibilities for $PQ$  such that $\mathscr{I}(PQ)$ satisfying the above properties are $\mathbb Z_{pq}\times \mathbb Z_p$ or $\langle a, b, c~|~ a^p=b^p=c^q=1, ab=ba, cac^{-1}=b,
cbc^{-1}= a^{-1}b^{l} \rangle$, where  $\bigl(\begin{smallmatrix}
  0 & -1\\ 1 & l
\end{smallmatrix} \bigr)$ has order $q$ in $GL_2(p)$, $q~|~(p+1)$. By a similar argument, we can show that the only possibilities for $PR$ are $\mathbb Z_{pr}\times \mathbb Z_p$ or $\langle a, b, c~|~ a^p=b^p=c^r=1, ab=ba, cac^{-1}=b,
cbc^{-1}= a^{-1}b^{k} \rangle$, where $\bigl(\begin{smallmatrix}
  0 & -1\\ 1 & k
\end{smallmatrix} \bigr)$ has order $r$ in $GL_2(p)$, $q~|~(p+1)$.
\begin{itemize}
\item  Suppose either $PQ$ or $PR$ is isomorphic to the  respective first group, with out loss of generality, let us take $PQ$. Then $PQ$ has at least three subgroups $H_iQ$, $i=1$, 2, 3 of order $pq$. Here $PQ$, $PR$, $QR$, $H_1Q$, $H_2Q$, $H_3Q$ intersect with each other non-trivially; $H_1$ is a subgroup of $PQ$, $PR$, $H_1Q$, $P$. It follows that these eight subgroups forms a subgraph of $\mathscr{I}(G)$, which is isomorphic to $\mathcal A_1$ as shown in Figure~\ref{fig:f2}. Therefore, $\gamma(\mathscr{I}(G))>1$ and $\overline{\gamma}(\mathscr{I}(G))>1$.

\item Now assume that $PQ$ and $PR$ are isomorphic to the respective second group.

\hspace{.7cm} Suppose that $PQ$ is the only subgroup of order $p^2q$, but the subgroup  of order $p^2r$ is not unique. If $n_r(G)=1$, then $R$ is normal in $G$. Then $H_1Q$, $H_1R$, $H_1QR$  are subgroups of $G$ of order $pq$, $pr$, $pqr$ respectively, where $H_1$ is subgroup of $P$ of order $p$. Here $PQ$, $PR$, $P$, $H_1Q$, $H_1R$, $H_1QR$ intersect non-trivially; $Q$ is a subgroup of $PQ$, $H_1Q$, $H_1QR$, $QR$. It follows that $\mathscr{I}(G)$ has a subgraph isomorphic to $\mathcal A_1$ as shown in Figure~\ref{fig:f2}. If $n_r(G)\neq 1$, then by Sylow's theorem $n_r(G)= p$, $q$, $pq$, $p^2$ or $p^2q$. Here $n_r(G)=p$, $q$, $pq$ are not possible, since $G$ has at least $2p^2$ subgroups of order $r$. Assume that $n_r(G)=p^2q$. Then for $q\geq 5$, $G$ has at least four subgroups of order $p^2r$, let them be $PR_1$, $PR_2$, $PR_3$, $PR_4$. It follows that $\mathscr{I}(G)$ contains $K_{5,4}$ as a subgraph with bipartition $X:=\{PQ$, $PR_1$, $PR_2$, $PR_3$, $PR_4\}$ and $Y:=\{P$, $H_1$, $H_2$, $H_3\}$, so $\gamma(\mathscr{I}(G))>1$, $\overline{\gamma}(\mathscr{I}(G))>1$, where $H_i$, $i=2$, 3 are subgroups of $P$ of order $p$. Also $PQ$, $P$, $H_1$, $PR_1$, $PR_2$ intersect non-trivially and so they form $K_5$ as a subgraph of $\mathscr{I}(G)$. If $p>2$, $q=3$, then $G$ has at least three subgroups of order $p^2r$ and so $\mathscr{I}(G)$ contains $K_{5,4}$ as a subgraph with bipartition $X:=\{PQ$, $PR_1$, $PR_2$, $PR_3$, $P\}$ and $Y:=\{H_1$, $H_2$, $H_3$, $H_4\}$, so $\gamma(\mathscr{I}(G))>1$, $\overline{\gamma}(\mathscr{I}(G))>1$, where $H_4$ is a  subgroup of $P$ of order $p$. Also $PQ$, $P$, $H_1$, $PR_1$, $PR_2$ intersect non-trivially and so they form $K_5$ as a subgraph of $\mathscr{I}(G)$. If $p=2$, $q=3$, then $n_r(G)=12$, which implies $r=11$. Thus we have $n_p(G)=1$, $n_q(G)=4$, $n_r(G)=12$, and this implies that $G$ has subgroup $QR$ of order $33$. Suppose $QR$ is cyclic, then $n_q(G)=n_r(G)$, which is not possible. Suppose $QR$ is non-cyclic, then $G$ has at least 11 subgroups of order $3$, which is also not possible. If $q=2$, $p\geq 5$, then $\mathscr{I}(G)$ contains $K_{5,4}$ as a subgraph with bipartition $X:=\{PQ$, $PR_1$, $PR_2$, $P\}$ and $Y:=\{H_1$, $H_2$, $H_3$, $H_4$, $H_5\}$, so $\gamma(\mathscr{I}(G))>1$, $\overline{\gamma}(\mathscr{I}(G))>1$, where $H_i$, $i=2$, 3 of $P$ of order $p$. Also $PQ$, $P$, $H_1$, $PR_1$, $PR_2$ intersect non-trivially and so they form $K_5$ as a subgraph of $\mathscr{I}(G)$. If $q=2$, $p=3$, then $n_r(G)=18$, which implies $r=17$. Here $n_r(G)=18$, $n_q(G)=9$ and by a similar argument as above, we can show that such a group does not exist. If $n_r(G)=p^2$, then $G$ has a unique subgroup of order $p^2r$, which is a contradiction to our assumption.

\hspace{.7cm} Suppose the subgroup of order $p^2q$ and $p^2r$ in $G$ are not unique, then we can use the previous argument to show $\gamma(\mathscr{I}(G))>1$, $\overline{\gamma}(\mathscr{I}(G))>1$ and $\mathscr{I}(G)$ contains $K_5$.

\hspace{.7cm} Suppose the subgroup of order $p^2q$ and $p^2r$ in $G$ are unique, then the subgroup structure of  $G$ is as follows: $G$ has unique subgroup of order $p$; $G$ has $p+1$ subgroups of order $p$, let them be $H_i$, $i=1$, 2, $\ldots$, $p+1$.
 Now $PQ$ has $p^2$ subgroups of order $q$, say $Q_i$, $i=1$, $\ldots$, $p^2$ and $PR$  has $p^2$ subgroups of order $r$, say $R_i$, $i=1$, $\ldots$, $p^2$. Now we show that the subgroup of order $qr$ in $G$ is not unique. If it is  unique, then this subgroup is isomorphic to one of $\mathbb Z_{qr}$, $\mathbb Z_q\rtimes \mathbb Z_r$ or $\mathbb Z_r\rtimes \mathbb Z_q$. But in either case, $G$ has a unique subgroup of order either $q$ or $r$, which is not possible.

 Note that the normalizer in $G$ of $Q$, $N_G(Q)$ is a subgroup of $G$ of order $qr$.  Similarly consider $R$, then $N_G(R)$
 is a subgroup of G of order $qr$. If $N_G(Q)$ is non-abelian, then necessarily $r< q$.  Likewise if $N_G(R)$ is non-abelian, then necessarily $q< r$.  As these inequalities cannot both be true, it follows that $G$ has an abelian (and hence cyclic) subgroup of order $qr$, and we may choose $Q$ and $R$ so that $N_G(Q)=N_G(R)= Q\times R$ is this cyclic subgroup. It follows that $G$ has $p^2$ subgroups of order $qr$, let them be $L_i$, $i=1, \ldots , p^2$. This completes the  subgroup structure of $G$.

 By the presentation of the subgroup of order $p^2 q$, an element $c$ of order $q$ acts on the $p$-Sylow subgroup $P$
via a matrix of determinant 1 which is not  $- 1$.  Hence $q$ is not 2.  Likewise $r$ is not 2.  So, $Q \times R$ is uniquely determined up to conjugacy in $GL_2(p)$.  Hence, $G$ is uniquely determined up to isomorphism.

 \hspace{.7cm} A presentation of this group $G$ is $\langle a, b, c~|~a^p=b^p=c^{qr}=1, ab=ba, cac^{-1}=b, cbc^{-1}=ab^l\rangle$, where $l$ is an integer with $\bigl(\begin{smallmatrix}
   0 & -1\\ 1 & l
 \end{smallmatrix} \bigr)$ has order $qr$ in $GL_2(p)$, $qr~|~(p+1)$.
 The structure of  $\mathscr{I}(G)$ is shown in Figure~\ref{fig:int f111}. Note that $\mathscr{I}(G)$ is $K_5$-free. If $p\geq 7$, then $\mathscr{I}(G)$ contains $K_{3,7}$ as a subgraph with bipartition $X:=\{PQ, PR, P\}$ and $Y:=\{H_1, H_2, H_3, H_4, H_5, H_6, H_7\}$, so $\gamma(\mathscr{I}(G))>1$ and $\overline{\gamma}(\mathscr{I}(G))>1$. If $p = 5$, then $\mathscr{I}(G)$ is toroidal and the corresponding toroidal embedding is shown in Figure~\ref{fig:int f11}.
Moreover, $\mathscr{I}(G)$ contains $K_{3,5}$ as a subgraph with bipartition $X:=\{PQ, PR, P\}$ and $Y:=\{H_1, H_2, H_3, H_4, H_5\}$, so $\overline{\gamma}(\mathscr{I}(G))>1$.
 Since $qr | (p+1)$, so $p\leq 3$ is not possible.

 \begin{figure}[ ht ]
 	\begin{center}
 		\includegraphics[scale=.8]{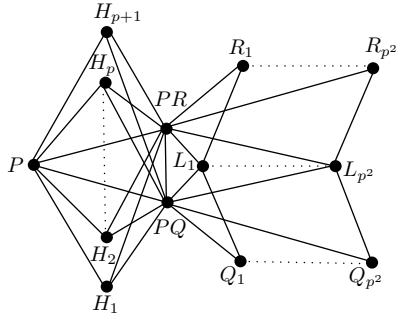}
 		\caption{The structure of $\mathscr{I}(G)$.}
 		\label{fig:int f111}
 	\end{center}
 \end{figure}

 \begin{figure}[ ht ]
 	\begin{center}
 		\includegraphics[scale=.8]{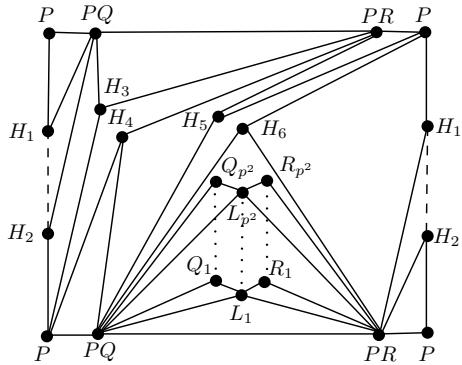}
 		\caption{An embedding of $\mathscr{I}(G)$ in the torus.}
 		\label{fig:int f11}
 	\end{center}
 \end{figure}


\end{itemize}
\end{enumerate}

\noindent \textbf{Subcase 1b:} Suppose $P$ is non-normal in $PQ$, then  $Q$ is normal in $PQ$. So, by Theorem~\ref{501}(iii), $Q$ is normal in $G$. Then $H_1Q$, $H_1R$, $H_1QR$  are subgroups of $G$ of order $pq$, $pr$, $pqr$ respectively, where $H_1$ is subgroup of $P$ of order $p$. Here $PQ$, $PR$, $P$, $H_1Q$, $H_1R$, $H_1QR$ intersect non-trivially; $Q$ is a subgroup of $PQ$, $H_1Q$, $H_1QR$, $QR$. It follows that $\mathscr{I}(G)$ has a subgraph isomorphic to $\mathcal A_1$ as shown in Figure~\ref{fig:f2}.

\noindent \textbf{Case 2:} A subgroup of $G$ of index $q$ is normal in $G$: With out loss of generality, we may assume that,  $PR$ is normal in  $G$. Here we can use a similar argument as in Subcase 1a, by taking $R$ instead of $Q$ and end up with the same group which we have obtained in Subcase 1a, whose intersection graph is toroidal, non-projective-planar and $K_5$-free.

\noindent\textbf{Case 3:} A subgroup of $G$ of index $p$ is normal in $G$: With out loss of generality, we may assume that, $H_1QR$ is normal in $G$, where $H_1$ is a subgroup of $P$ of order $p$. Then by Theorem~\ref{501}(iii), $H_1$ is normal in $G$. So, $H_1Q$, $H_1R$, $H_1QR$  are subgroups of $G$. Here $PQ$, $PR$, $P$, $H_1Q$, $H_1R$, $H_1QR$ intersect non-trivially; $Q$ is a subgroup of $PQ$, $H_1Q$, $H_1QR$, $QR$. It follows that $\mathscr{I}(G)$ has a subgraph isomorphic to $\mathcal A_1$ as shown in Figure~\ref{fig:f2}.

Putting together all the cases,  the result follows.
\end{proof}

\begin{pro}\label{intersection graph t11}
Let $G$ be a non-abelian solvable group of order $p^3qr$, where $p$, $q$, $r$ are distinct primes. Then $\gamma(\mathscr{I}(G))>1$,  $\overline{\gamma}(\mathscr{I}(G))>1$ and $\mathscr{I}(G)$ contains $K_5$.
\end{pro}
\begin{proof}
Since $G$ is solvable, so it has a Sylow basis $ \{P , Q, R \}$, where $p, Q, R$ are Sylow $p,q,r$-subgroup of $G$ respectively. Then $PQ$, $PR$ are subgroups of $G$. Suppose $\gamma(\mathscr{I}(PQ))>1$,
$\overline{\gamma}(\mathscr{I}(PQ))>1$ and
$\mathscr{I}(PQ)$ contains $K_5$,
then the same holds for $\mathscr{I}(G)$ also. So it is enough to consider the cases when $\mathscr{I}(PQ)$ is one of planar, toroidal, projective-planar,
$K_5$-free, $C_3$-free and bipartite. By Theorem~\ref{1000},
Propositions~\ref{intersecting graph t100}, \ref{intersecting graph t1} and \ref{intersecting graph t5}, $PQ\cong \mathbb Z_{p^3q}$. By a similar argument, we can show that the only possibility of $PR$ is $\mathbb Z_{p^2r}$.
Let $H_1$, $H_2$, $H_3$, $H_4$, $H_5$ be subgroups of $PQ$ of orders $p$, $p^2$, $p^3$, $pq$, $p^2q$ respectively; $H$ be a subgroup of $PR$ of order $p^2r$. Here $H_1$ is a
subgroup of $H$, $PQ$, $PR$, $H_i$, $i=2$ to 5. It follows that, these five subgroups intersect with each other non-trivially and so $\mathscr{I}(G)$ contains
$K_8$ as a subgraph. Hence, $\gamma(\mathscr{I}(G))>1$ and $\overline{\gamma}(\mathscr{I}(G))>1$.
\end {proof}

\begin{pro}\label{intersection graph t12}
Let $G$ be a non-abelian solvable group of order $p^2q^2r$, where $p$, $q$, $r$ are distinct primes. Then $\gamma(\mathscr{I}(G))>1$,  $\overline{\gamma}(\mathscr{I}(G))>1$ and $\mathscr{I}(G)$ contains $K_5$.
\end{pro}
\begin{proof}
Since $G$ is a solvable group, it has subgroups of order $p^2r$ and $q^2r$. Now we split the proof in to the following cases:\\
 \noindent \textbf{Case 1:} A subgroup of $G$ of order either $p^2r$ or $q^2r$ is unique: Then $G$ has a subgroup of order either $p^2qr$ or $pq^2r$,  let it be $H$, and without loss of generality, we assume that $|H|=p^2qr$. Suppose $\gamma(\mathscr{I}(H))>1$,
$\overline{\gamma}(\mathscr{I}(H))>1$ and
$\mathscr{I}(H)$ contains $K_5$,
then the same holds for $\mathscr{I}(G)$ also. So it is enough to consider the cases when $\mathscr{I}(H)$ is one of planar, toroidal, projective-planar,
$K_5$-free, $C_3$-free and bipartite. By Theorem~\ref{1000}, Propositions~\ref{intersecting graph t100}, \ref{intersecting graph t1} and \ref{intersecting graph t10}, $H\cong \langle a, b, c~|~ a^p=b^p=c^{qr}=1, ab=ba, cac^{-1}=b, cbc^{-1}=ab^l\rangle$, where $\bigl(\begin{smallmatrix}
 0 & -1\\ 1 & l
\end{smallmatrix} \bigr)$ has order $qr$ in $GL_2(p)$, $qr|(p+1)$. So by Figure~\ref{fig:int f111}, $H$ together with its proper subgroups forms $K_{3,7}$ as a subgraph of $\mathscr{I}(G)$ and it contains $K_5$ as a subgraph.

\noindent \textbf{Case 2:} Subgroups of $G$ of order $p^2r$, $q^2r$ are not unique:\\
Since $G$ is a solvable group, so it has a subgroup, say $N$ of prime index. Now we need to consider the following subcases:

\noindent \textbf{Subcase 2a:} Let $[G:N]=p$ or $q$, without loss of generality, we assume that $[G:N]=q$, and so $|N|=p^2qr$. Then by the argument mentioned in Case 1, $\gamma(\mathscr{I}(G))>1$, $\overline{\gamma}(\mathscr{I}(G))>1$ and $\mathscr{I}(G)$ contains $K_5$ as a subgraph.

\noindent \textbf{Subcase 2b:} Let $[G:N]=r$. Then $|N|=p^2q^2$. Suppose $\gamma(\mathscr{I}(N))>1$,
$\overline{\gamma}(\mathscr{I}(N))>1$ and
$\mathscr{I}(N)$ contains $K_5$,
then the same holds for $\mathscr{I}(G)$ also. So it is enough to consider the cases when $\mathscr{I}(N)$ is one of planar, toroidal, projective-planar,
$K_5$-free, $C_3$-free and bipartite. By Theorem~\ref{1000}, Propositions~\ref{intersecting graph t100}, \ref{intersecting graph t1} and \ref{intersecting graph t6}, $N\cong \mathbb Z_{p^2q^2}$ or $\langle a, b, c~|~ a^p=b^p=c^{q^2}=1, ab=ba, cac^{-1}=b, cbc^{-1}=ab^l\rangle$, where $\bigl(\begin{smallmatrix}
 0 & -1\\ 1 & l
\end{smallmatrix} \bigr)$ has order $q^2$ in $GL_2(p)$, $q^2|(p+1)$. If $N\cong \mathbb Z_{p^2q^2}$, then $N$ together with its proper subgroups forms a subgraph isomorphic to $\mathcal A_1$ shown in Figure~\ref{fig:f2}. Therefore, $\gamma(\mathscr{I}(G))>1$ and $\overline{\gamma}(\mathscr{I}(G))>1$.

If $N$ isomorphic to the second group, then it has a subgroup of order $p^2q$, let it be $I$.
Suppose $G$ has unique subgroup of order $p^2r$, then $G$ has a subgroup of order $p^2qr$, let it be $H$. Suppose $\gamma(\mathscr{I}(H))>1$,
$\overline{\gamma}(\mathscr{I}(H))>1$ and
$\mathscr{I}(H)$ contains $K_5$,
then the same holds for $\mathscr{I}(G)$ also. So it is enough to consider the cases when $\mathscr{I}(H)$ is one of planar, toroidal, projective-planar,
$K_5$-free, $C_3$-free and bipartite. By Theorem~\ref{1000}, Propositions~\ref{intersecting graph t100}, \ref{intersecting graph t1} and \ref{intersecting graph t10}, $H\cong \langle a, b, c~|~a^p=b^p=c^{qr}=1, ab=ba, cac^{-1}=b^{-1}, cbc^{-1}=ab^l\rangle$, where $\bigl(\begin{smallmatrix}
 0 & -1\\ 1 & l
\end{smallmatrix} \bigr)$ has order $qr$ in $GL_2(p)$, $qr|(p+1)$. By Figure~\ref{fig:int f111}, $H$ together with its subgroups forms $K_{3,7}$ as a subgraph of $\mathscr{I}(G)$ and it contains $K_5$. Suppose the subgroup of order $p^2r$ is not unique, then $G$ has at least two subgroups of order $p^2r$, let them be $PR_1$, $PR_2$. Let $L_1$ be a subgroup of $Q$ of order $q$. Here $H_1$ is a subgroup of $P$,  $PR_1$, $PR_2$, $N$; $L_1$ is a subgroup of $Q$, $N$, $I$, $QR$. It follows that $\mathscr{I}(G)$ has a subgraph isomorphic to $\mathcal B_2$ as shown in Figure~\ref{fig:f2}. Therefore, $\gamma(\mathscr{I}(G))>1$, $\overline{\gamma}(\mathscr{I}(G))>1$ and $\mathscr{I}(G)$ contains $K_5$ as a subgraph.

 Proof follows by combining all the cases together.
\end{proof}

\begin{pro}\label{intersection graph t13}
Let $G$ be a non-abelian solvable group of order $p^\alpha q^\beta r^\delta$, where $p$, $q$, $r$ are distinct primes and $\alpha+\beta+\delta=6$. Then $\gamma(\mathscr{I}(G))>1$,  $\overline{\gamma}(\mathscr{I}(G))>1$ and $\mathscr{I}(G)$ contains $K_5$.
\end{pro}
\begin{proof}
Here $|G|=p^4qr$, $p^3q^2r$ or $p^2q^2r^2$. First let us assume that  $|G|$ is either $p^4qr$ or $p^3q^2r$. Since $G$ is solvable, it has a subgroup of order either $p^4q$ or $p^3q^2$ respectively, let it be $H$. Suppose $\gamma(\mathscr{I}(H))>1$,
$\overline{\gamma}(\mathscr{I}(H))>1$ and
$\mathscr{I}(H)$ contains $K_5$,
then the same holds for $\mathscr{I}(G)$ also. So it is enough to consider the cases when $\mathscr{I}(H)$ is one of planar, toroidal, projective-planar,
$K_5$-free, $C_3$-free and bipartite. By Theorem~\ref{1000}, Propositions~\ref{intersecting graph t100}, \ref{intersecting graph t1}, \ref{intersecting graph t5} and \ref{intersecting graph t7}, we have $H\cong \mathbb Z_{p^4q}$. Here $H$ together with its proper subgroups forms $K_8$, so $\gamma(\mathscr{I}(G))>1$ and $\overline{\gamma}(\mathscr{I}(G))>1$.

If $|G|=p^2q^2r^2$, then $G$ has a normal subgroup $N$ with prime index, without loss of generality, say $r$. Since $|N|=p^2q^2r$, and so by Theorem~\ref{1000}, Propositions~\ref{intersecting graph t100}, \ref{intersecting graph t1} and \ref{intersection graph t12}, $\gamma(\mathscr{I}(N))>1$, $\overline{\gamma}(\mathscr{I}(N))>1$ and $\mathscr{I}(N)$ contains $K_5$. Therefore, the same holds for $\mathscr{I}(G)$ also.
Hence $\gamma(\mathscr{I}(G))>1$, $\overline{\gamma}(\mathscr{I}(G))>1$  and  $\mathscr{I}(G)$ contains $K_5$ as a subgraph. This completes the proof.
\end{proof}

\begin{pro}\label{intersection graph t14}
Let $G$ be a non-abelian solvable group of order $p^\alpha q^\beta r^\delta$, where $p$, $q$, $r$ are distinct primes and $\alpha+\beta+\delta\geq 7$. Then $\gamma(\mathscr{I}(G))>1$,  $\overline{\gamma}(\mathscr{I}(G))>1$ and $\mathscr{I}(G)$ contains $K_5$.
\end{pro}
\begin{proof}
Without loss of generality, we assume that $\alpha\geq \beta\geq \delta$. Since $G$ is solvable, $G$ has a Sylow basis
containing $P$, $Q$, $R$, where $P$, $Q$, $R$ are Sylow $p,q,r$-subgroups of $G$ respectively and so $PQ$ is a subgroup of $G$. By Theorem~\ref{1000}, Propositions~\ref{intersecting graph t100},~\ref{intersecting graph t1},~\ref{intersecting graph t5} and \ref{intersecting graph t7},
$\gamma(\mathscr{I}(PQ))>1$, $\overline{\gamma}(\mathscr{I}(PQ))>1$, $\mathscr{I}(PQ)$ contains $K_5$, and so $\gamma(\mathscr{I}(G))>1$, $\overline{\gamma}(\mathscr{I}(G))>1$ and $\mathscr{I}(G)$ contains $K_5$ as a subgraph.
\end{proof}

\begin{pro}\label{intersecting graph t9}
 If $G$ is a solvable group, whose order has more than three distinct prime factors, then $\gamma(\mathscr{I}(G))>1$, $\overline{\gamma}(\mathscr{I}(G))>1$ and $\mathscr{I}(G)$ contains $K_5$.
\end{pro}
\begin{proof}
 Since $G$ is solvable, it has a Sylow basis containing $P$, $Q$, $R$, $S$, where $P$, $Q$, $R$, $S$ are Sylow $p$, $q$, $r$, $s$-subgroups of $G$ respectively. Then $P$, $PQ$,
$PR$, $PS$, $PQR$, $PQS$, $PRS$, $QRS$ are subgroups of $G$. Here $P$ is a subgroup of $PQ$,
$PR$, $PS$, $PQR$, $PQS$, $PRS$; $QRS$, $PR$, $PS$, $PQR$, $PQS$, $PRS$, $QRS$ intersect with each other nontrivially.
It follows that $\mathscr{I}(G)$ has a subgraph $\mathcal A_1$ as shown in Figure~\ref{fig:f2}, so  $\gamma(\mathscr{I}(G))>1$, $\overline{\gamma}(\mathscr{I}(G))>1$ and $\mathscr{I}(G)$ contains $K_5$.
\end{proof}

\subsection{Finite non-solvable groups}\label{sec:4b}
If $G$ is a group and $N$ is a normal subgroups of $G$, then $\mathscr{I}(G/N)$ is isomorphic (as a graph) to a subgraph of $\mathscr{I}(G)$.
It is well known that any non-solvable group has a simple group as a sub-quotient and every simple group has a minimal simple group as a sub-quotient.
Therefore, if we can show that the minimal simple groups have non-toroidal (non-projective-planar) intersection graphs,
then the intersection graph of a non-solvable group is non-toroidal (resp., non-projective-planar).

Recall that $SL_m(n)$ is the group of $m \times m$ matrices having determinant 1, whose entries are lie in a field with $n$ elements and that
$L_m(n)=SL_m(n)/ H$, where $H=\{kI| k^m=1 \}$. For any prime $q > 3$, the Suzuki group is denoted by $Sz(2^q)$

%
%

\begin{lemma}\label{intersecting graph l1}
If $n>2$, then $\gamma(\mathscr{I}(D_{4n}))>1$, $\overline{\gamma}(\mathscr{I}(D_{4n}))>1$ and $\mathscr{I}(G)$ contains $K_5$.
\end{lemma}
\begin{proof}

Here $H_1:=\langle a \rangle$, $H_2:=\langle a^n, b\rangle$, $H_3:=\langle a^n, ba\rangle$, $H_4:=\langle a^n, ba^2\rangle$, $H_5:=\langle a^n\rangle$,
$H_6:=\langle a^{2n},b\rangle$, $H_7:=\langle a^{2n},ba\rangle$, $H_8:=\langle a^{2n},ba^2\rangle$, $H_9:=\langle a^{2n},ba^3\rangle$,
$H_{10}:=\langle a^{2n}\rangle$ are proper subgroups of $D_{4n}$. Also $H_5$ is a subgroup of $H_i$, $i=1$, 2, 3, 4; $H_{10}$ is a subgroup of $H_i$,
$i=6$, 7, 8, 9. It follows that $\mathscr{I}(D_{4n})$ has a subgraph $\mathcal A_1$ as shown in Figure~\ref{fig:f2}, so $\gamma(\mathscr{I}(D_{4n}))>1$,
$\overline{\gamma}(\mathscr{I}(D_{4n}))>1$.
\end{proof}

\begin{pro}\label{intersecting graph t11}
 If $G$ is a finite non-solvable group, then $\gamma(\mathscr{I}(G))>1$, $\overline{\gamma}(\mathscr{I}(G))>1$ and $\mathscr{I}(G)$ contains $K_5$.
\end{pro}
 \begin{proof} The classification of minimal simple groups is given in \cite[Corollary 1]{thomp}.
As mentioned above, to prove the result, it is enough to show that these minimal simple groups are non-toroidal, not projective-planar and contains $K_5$.
 Now we investigate each of these groups. Her we will denote the image of a matrix $A$ in $L_m(n)$ by $\overline{A}$.

\noindent\textbf{Case 1:}  $G \cong L_2(q^p)$, where $q=2,3$ and $p$ is any prime.
If $p=2$, then the only non-solvable group is $ L_2(4)$. Also $ L_2(4) \cong A_5$ (see \cite{atlas}). $A_5$ has five copies of $A_4$ and any two $A_4$ in $A_5$ have
non-trivial intersection, for otherwise $|A_4A_4|=144$, which is not possible. Also $H_1:=\langle (12, 34)\rangle$, $H_2:=\langle (12,34), (13, 24)\rangle$,
$H_3:=\langle (12,34), (12354)\rangle$, $H_4:=\langle (12,34), (12453)\rangle$, $H_5:=\langle (12,34),(345)\rangle$ are proper subgroups of $A_5$. Here
$H_1$ is a subgroup of $H_i$, for every $i=1$, 2, 3, 4. It follows that $\mathscr{I}(A_5)$ has a subgraph $\mathcal B_1$ as in Figure~\ref{fig:int f3}, so
$\gamma(\mathscr{I}(G))>1$ and $\overline{\gamma}(\mathscr{I}(G))>1$.


If $p> 2$, then $ L_2(q^p)$ contains a subgroup isomorphic to $(\mathbb Z_q)^p$,
 namely the subgroup of matrices of the form $\overline{
\bigl( \begin{smallmatrix}
1 & a\\
0 & 1
\end{smallmatrix}\bigr)
}$ with $a \in \mathbb F_{q^p}$.
By Proposition~\ref{intersecting graph t1}, $\gamma(\mathscr{I}((\mathbb Z_q)^p ))>1$, and $\overline{\gamma}(\mathscr{I}((\mathbb Z_q)^p ))>1$, $p >2$
and $\mathscr{I}((\mathbb Z_q)^p )$ contains $K_5$.

\noindent\textbf{Case 2:} $G \cong  L_3(3)$. In $SL_3(3)$ the only matrix in the subgroup $H$ is the identity matrix, so $L_3(3) \cong SL_3(3)$.
Let us consider the subgroup consisting of matrices of the form
$\left(\begin{smallmatrix}
1 & a & b\\
0 & 1 & c\\
0 & 0 & 1
\end{smallmatrix}\right)$
with $a, b, c \in \mathbb F_3$. This subgroup is isomorphic to the group
$(\mathbb Z_p \times \mathbb Z_p) \rtimes \mathbb Z_p$ with $p=3$.
 By Proposition~\ref{intersecting graph t2}, $\gamma(\mathscr{I}((\mathbb Z_p \times \mathbb Z_p) \rtimes \mathbb Z_p))>1$,
$\overline{\gamma}(\mathscr{I}((\mathbb Z_p \times \mathbb Z_p) \rtimes \mathbb Z_p))>1$ and
$\mathscr{I}((\mathbb Z_p \times \mathbb Z_p) \rtimes \mathbb Z_p)$ contains $K_5$.

\noindent\textbf{Case 3:} $G \cong L_2(p)$, where $p$ is any prime exceeding $3$ such that $p^2  +1 \equiv 0~(\text{mod}~ 5)$. We have to consider two subcases:

\noindent\textbf{Subcase 3a:} $p \equiv 1  ~(\text{mod}~ 4)$. It is shown in \cite[p.~222]{boh-reid} that  $L_2(p)$
has a  subgroup isomorphic to $D_{p-1}$. So by Lemma~\ref{intersecting graph l1}, $\gamma(\mathscr{I}(D_{p-1}))>1$,
$\overline{\gamma}(\mathscr{I}(D_{p-1}))>1$ and
$\mathscr{I}(D_{p-1})$ contains $K_5$.

If $p=5$, then $L_2(5) \cong A_5 \cong L_2(4)$. By Case 1,
$\gamma(\mathscr{I}(A_5))>1$, $\overline{\gamma}(\mathscr{I}(A_5))>1$ and $\mathscr{I}(A_5)$ contains $K_5$.

\noindent\textbf{Subcase 3b:} $p \equiv 3 ~(\text{mod}~ 4)$.  $L_2(p)$
has  a subgroup isomorphic to $D_{p+1}$\cite[p.~222]{boh-reid}.  By Lemma~\ref{intersecting graph l1}, $\gamma(\mathscr{I}(D_{p+1}))>1$,
$\overline{\gamma}(\mathscr{I}(D_{p+1}))>1$ and
$\mathscr{I}(D_{p+1})$ contains $K_5$.
If $p=7$, then $S_4$ is a maximal subgroup of $L_2(7)$\cite{atlas}. Also by Theorem~\ref{intersecting graph t5}, $\gamma(\mathscr{I}(S_4))>1$,
$\overline{\gamma}(\mathscr{I}(S_4))>1$ and
$\mathscr{I}(S_4)$ contains $K_5$.


\noindent\textbf{Case 4:}  $G \cong Sz(2^q)$, where $q$ is any odd prime. Then  $ Sz(2^q)$ has a subgroup isomorphic to
$(\mathbb Z_2)^q, q \geq 3$ \cite[p.~466]{goren}. But by Proposition~\ref{intersecting graph t1}, $\gamma(\mathscr{I}((\mathbb Z_2)^q))>1$,
$\overline{\gamma}(\mathscr{I}((\mathbb Z_2)^q))>1$, $q \geq 3$ and
$\mathscr{I}((\mathbb Z_2)^q)$ contains $K_5$.

Combining all these cases together, the proof follows.
\end{proof}

\section{Main results}\label{sec: 5}
By combining all the results obtained in Sections~\ref{sec:3} and \ref{sec:4} above, we have the following general main results,
which classifies the finite groups whose intersection graphs of subgroups are toroidal or projective-planar, and classifies the finite non-cyclic groups whose intersection graphs of subgroups are one of $K_5$ free, $C_3$-free, acyclic or bipartite.

\begin{thm}\label{intersecting graph t12}
 Let $G$ be a finite group and $p$, $q$, $r$ be distinct primes. Then
\begin{enumerate}[{\normalfont (1)}]
\item $\mathscr{I}(G)$ is toroidal if and only if $G$ is isomorphic to one of the following:
\begin{enumerate}[\normalfont (a)]
\item $\mathbb Z_{p^\alpha} (\alpha=6,7,8)$, $\mathbb Z_{p^\alpha q}(\alpha=3,4)$, $\mathbb Z_{p^2q^2}$, $\mathbb Z_{p^2qr}$, $\mathbb Z_9\times \mathbb Z_3$, $\mathbb Z_{25}\times \mathbb Z_5$, $\mathbb Z_{3q}\times \mathbb Z_3$, $M_{p^3} (p= 3, 5)$, $M_{16}$, $\mathbb Z_3\rtimes \mathbb Z_4$, $\mathbb Z_5\rtimes \mathbb Z_4$, $\mathbb Z_9\rtimes \mathbb Z_2$, $\mathbb Z_{25}\rtimes \mathbb Z_2$;
\item $\langle a, b, c~|~ a^5=b^5=c^6=1, ab=ba, cac^{-1}=b,
cbc^{-1}= ab^{l} \rangle$, where $\bigl(\begin{smallmatrix}
  0 & -1\\ 1 & l
\end{smallmatrix} \bigr)$ has order $6$ in $GL_2(5)$.
\end{enumerate}
\item $\mathscr{I}(G)$ is projective-planar if and only if $G$ is isomorphic to one of
 $\mathbb Z_{p^\alpha}$$(\alpha=6, 7)$, $\mathbb Z_{p^3q}$, $\mathbb Z_9\times \mathbb Z_3$, $\mathbb Z_{3q}\times \mathbb Z_3$, $M_{27}$, $\mathbb Z_3\rtimes \mathbb Z_4$, $\mathbb Z_9\rtimes \mathbb Z_2$.
\end{enumerate}
\end{thm}

\begin{thm}\label{intersecting graph t233}
 Let $G$ be a finite non-cyclic group and $p$, $q$, $r$ be distinct primes. Then
\begin{enumerate}[{\normalfont (1)}]
\item $\mathscr{I}(G)$ is $K_5$-free if and only if $G$ is isomorphic to one of the following groups:
\begin{enumerate}[{\normalfont(a)}]
\item $\mathbb Z_p \times \mathbb Z_p$, $\mathbb Z_4 \times \mathbb Z_2$, $\mathbb Z_6\times \mathbb Z_2$, $Q_8$, $M_8$, $\mathbb Z_q \rtimes \mathbb Z_p$, $\mathbb Z_q \rtimes_2 \mathbb Z_{p^2}$, $A_4$;
\item $\langle a, b, c~|~ a^p=b^p=c^q=1, ab=ba, cac^{-1}=b,
cbc^{-1}= a^1b^{l} \rangle$, where $\bigl(\begin{smallmatrix}
  0 & -1\\ 1 & l
\end{smallmatrix} \bigr)$ has order $q$ in $GL_2(p)$,  $q|(p+1)$;
\item $\langle a,b,c~|~a^p=b^p=c^{q^2}=1, ab=ba, cac^{-1}=b^{-1}, cbc^{-1}=a^1b^l\rangle$, where
 $\bigl(\begin{smallmatrix}
  0 & -1\\ 1 & l
\end{smallmatrix} \bigr)$ has order $q^2$ in $GL_2(p)$,  $q^2|(p+1)$;
\item $\langle a,b,c~|~a^p=b^q=c^r=1, b^{-1}ab=a^\mu, c^{-1}ac=a^v, bc=cb\rangle$, where $r$, $q$ are divisor of $p$ and $v$, $\mu\neq 1$;
\item $\langle a, b, c~|~ a^p=b^p=c^{qr}=1, ab=ba, cac^{-1}=b,
cbc^{-1}= ab^{l} \rangle$, where $\bigl(\begin{smallmatrix}
  0 & -1\\ 1 & l
\end{smallmatrix} \bigr)$ has order $qr$ in $GL_2(p)$, $qr|(p+1)$.
\end{enumerate}
\item The following are equivalent:
\begin{enumerate}[{\normalfont(a)}]
\item $G$ is isomorphic to one of the following:
$\mathbb Z_{p^\alpha}(\alpha=2,3)$, $\mathbb Z_{pq}$, $\mathbb Z_p \times \mathbb Z_p$, $\mathbb Z_q \rtimes \mathbb Z_p$, $A_4$ or $\langle a, b, c~|~ a^p=b^p=c^q=1, ab=ba, cac^{-1}=b,
cbc^{-1}= a^1b^{l} \rangle$, where $\bigl(\begin{smallmatrix}
  0 & -1\\ 1 & l
\end{smallmatrix} \bigr)$ has order $q$ in $GL_2(p)$, $q|(p+1)$;
\item $\mathscr{I}(G)$ is $C_3$-free;
\item $\mathscr{I}(G)$ is acyclic;
\item $\mathscr{I}(G)$ is bipartite.
\end{enumerate}
\end{enumerate}
\end{thm}

\begin{cor}\label{intersecting graph c1}
 Let $G$ be a finite group and $p,q,r$ are distinct primes. Then
 \begin{enumerate}[{\normalfont(1)}]
%
%
\item $\mathscr{I}(G)$ is unicyclic if and only if $G$ is either $\mathbb Z_{p^4}$ or $\mathbb Z_{p^2q}$;
\item $\mathscr{I}(G)\cong C_n$ if and only if $n=3$ and $G \cong \mathbb Z_{p^4}$;
\item $\mathscr{I}(G)\cong P_n$ if and only if $n=1$ and $G \cong \mathbb Z_{p^3}$;

\item $\mathscr{I}(G)$ is $C_5$-free if and only if $G$ is one of $\mathbb Z_{p^\alpha} (\alpha=2,3,4,5)$, $\mathbb Z_{p^\alpha q} (\alpha=1,2)$,
$\mathbb Z_p\times \mathbb Z_p$, $\mathbb Z_4\times \mathbb Z_2$, $Q_8$, $\mathbb Z_q\rtimes \mathbb Z_p$, $\mathbb Z_q\rtimes_2 \mathbb Z_{p^2}$, $A_4$,
$\langle a, b, c~|~ a^p=b^p=c^q=1, ab=ba, cac^{-1}=b,
cbc^{-1}= a^1b^{l} \rangle$, where $\bigl(\begin{smallmatrix}
  0 & -1\\ 1 & l
\end{smallmatrix} \bigr)$ has order $q$ in $GL_2(p)$, $q~|~(p+1)$ or $\langle a,b,c~|~a^p=b^p=c^{q^2}=1, ab=ba, cac^{-1}=b^{-1}, cbc^{-1}=a^1b^l\rangle$,
 where $\bigl(\begin{smallmatrix}
  0 & -1\\ 1 & l
\end{smallmatrix} \bigr)$ has order $q^2$ in $GL_2(p)$, $q^2|(p+1)$;
\item $\mathscr{I}(G)$ is $C_4$-free if and only if $G$ is one of $\mathbb Z_{p^\alpha}(\alpha=2,3,4)$, $\mathbb Z_{p^\alpha q}(\alpha= 1,2)$,
$\mathbb Z_p \times \mathbb Z_p$, $\mathbb Z_q \rtimes \mathbb Z_p$, $\mathbb Z_q \rtimes_2 \mathbb Z_{p^2}$, $A_4$ or
$\langle a, b, c~|~ a^p=b^p=c^q=1, ab=ba, cac^{-1}=b,
cbc^{-1}= a^1b^{l} \rangle$, where $\bigl(\begin{smallmatrix}
  0 & -1\\ 1 & l
\end{smallmatrix} \bigr)$ has order $q$ in $GL_2(p)$, $q|(p+1)$;

\item $\mathscr{I}(G)$ is $P_4$-free if and only if $G$ is one of $\mathbb Z_{p^\alpha}(\alpha=2,3,4,5)$, $\mathbb Z_{p^\alpha q}(\alpha= 1,2)$,
$Q_8$, $\mathbb Z_p \times \mathbb Z_p$, $\mathbb Z_q \rtimes \mathbb Z_p$ or
$\langle a, b, c~|~ a^p=b^p=c^q=1, ab=ba, cac^{-1}=b,
cbc^{-1}= a^1b^{l} \rangle$, where $\bigl(\begin{smallmatrix}
  0 & -1\\ 1 & l
\end{smallmatrix} \bigr)$ has order $q$ in $GL_2(p)$, $q|(p+1)$ or $A_4$;
\item $\mathscr{I}(G)$ is $P_3$-free if and only if $G$ is one of $\mathbb Z_{p^\alpha}(\alpha=2,3,4)$, $\mathbb Z_{pq}$,
$\mathbb Z_p \times \mathbb Z_p$, $\mathbb Z_q \rtimes \mathbb Z_p$, $A_4$ or
$\langle a, b, c~|~ a^p=b^p=c^q=1, ab=ba, cac^{-1}=b,
cbc^{-1}= a^1b^{l} \rangle$, where $\bigl(\begin{smallmatrix}
  0 & -1\\ 1 & l
\end{smallmatrix} \bigr)$ has order $q$ in $GL_2(p)$, $q|(p+1)$;
\item $\mathscr{I}(G)$ is $P_2$-free if and only if $G$ is one of $\mathbb Z_{p^\alpha}(\alpha=2,3)$, $\mathbb Z_{pq}$,
$\mathbb Z_p \times \mathbb Z_p$ or $\mathbb Z_q \rtimes \mathbb Z_p$;
\item $\mathscr{I}(G)$ is totally disconnected if and only if $G$ is isomorphic to one of $\mathbb Z_{p^2}$, $\mathbb Z_{pq}$, $\mathbb Z_p\times \mathbb Z_p$ or $\mathbb Z_q\rtimes \mathbb Z_p$;
\item $\mathscr{I}(G)$ is $K_{2,3}$-free if and only if $G$ is one of $\mathbb Z_{p^\alpha}(\alpha=2,3,4,5)$, $\mathbb Z_{p^\alpha q}(\alpha=1,2)$,
$\mathbb Z_{pqr}$,
$\mathbb Z_p\times \mathbb Z_p$, $\mathbb Z_4\times \mathbb Z_2$, $Q_8$, $\mathbb Z_q\rtimes Z_p$, $\mathbb Z_q\rtimes_2 \mathbb Z_{p^2}$, $A_4$ or
$\langle a, b, c~|~ a^p=b^p=c^q=1, ab=ba, cac^{-1}=b,
cbc^{-1}= a^1b^{l} \rangle$, where $\bigl(\begin{smallmatrix}
  0 & -1\\ 1 & l
\end{smallmatrix} \bigr)$ has order $q$ in $GL_2(p)$, $q|(p+1)$;
\item $\mathscr{I}(G)$ is $K_4$-free if and only if $G$ is one of
 $\mathbb Z_{p^\alpha} (\alpha=2,3,4)$, $\mathbb Z_{p^\alpha q} (\alpha=1,2)$, $\mathbb Z_{pqr}$,
$\mathbb Z_p\times \mathbb Z_p$, $\mathbb Z_q\rtimes_2 \mathbb Z_p$, $\mathbb Z_q\rtimes \mathbb Z_{p^2}$,  $A_4$ or
$\langle a, b, c~|~ a^p=b^p=c^q=1, ab=ba, cac^{-1}=b,
cbc^{-1}= a^1b^{l} \rangle$, where $\bigl(\begin{smallmatrix}
  0 & -1\\ 1 & l
\end{smallmatrix} \bigr)$ has order $q$ in $GL_2(p)$, $q|(p+1)$ ;
\item $\mathscr{I}(G)$ is $K_{1,4}$-free if and only if $G$ is one of $\mathbb Z_{p^\alpha} (\alpha=2,3,4,5)$, $\mathbb Z_{p^\alpha q} (\alpha=1,2)$,
$\mathbb Z_p\times \mathbb Z_p$, $Q_8$, $\mathbb Z_q\rtimes \mathbb Z_p$ or $A_4$;
\item $\mathscr{I}(G)$ is claw-free if and only if $G$ is one of $\mathbb Z_{p^\alpha}(\alpha=2,3,4)$, $\mathbb Z_{pq}$,
$\mathbb Z_p \times \mathbb Z_p$ or $\mathbb Z_q \rtimes \mathbb Z_p$;
\item The following are equivalent:
\begin{enumerate}[{\normalfont(a)}]
\item $G\cong \mathbb Z_{p^\alpha}(\alpha=2,3)$;
\item $\mathscr{I}(G)$ is a tree;
\item $\mathscr{I}(G)$ is a star graph;
\item $\mathscr{I}(G)$ is complete bipartite.
\end{enumerate}
\item $girth(\mathscr{I}(G))$ is $\infty$, if $G$ is one of $\mathbb Z_{p^\alpha}(\alpha=2,3)$, $\mathbb Z_{pq}$, $\mathbb Z_p \times \mathbb Z_p$, $\mathbb Z_q \rtimes \mathbb Z_p$, $A_4$ or $\langle a, b, c~|~ a^p=b^p=c^q=1, ab=ba, cac^{-1}=b,
cbc^{-1}= a^1b^{l} \rangle$, where $\bigl(\begin{smallmatrix}
  0 & -1\\ 1 & l
\end{smallmatrix} \bigr)$ has order $q$ in $GL_2(p)$, $q|(p+1)$;
otherwise $girth(\mathscr{I}(G))=3$.
\end{enumerate}
\end{cor}
\begin{proof}
Note that if the intersection graph of subgroups of a group has $K_5$ as a subgraph, then it is none of the following graphs: unicyclic, cycle, path, claw-free,  $C_5$-free,   $C_4$-free, $P_4$-free, $P_3$-free,  $P_2$-free, totally disconnected, $K_{2,3}$-free, $K_4$-free, $K_{1,4}$-free.
So to classify the finite groups whose intersection graph of subgroups is one of
outerplanar, unicyclic, claw-free, path, cycle, $C_4$-free, $C_5$-free, totally disconnected, $P_2$-free, $P_3$-free, $P_4$-free, $K_{1,4}$-free,
$K_{2,3}$-free,
it is enough to consider the finite groups whose intersection graph of subgroups are $K_5$-free. Thus, we need to investigate these properties only for groups
given in Theorems~\ref{l1}(2) and \ref{intersecting graph t233}(1).

\noindent By Theorem~\ref{l1}(2) and using \eqref{intersection graph e1}, \eqref{intersection graph e2}, the only groups $G$ such that $\mathscr{I}(G)$  is unicyclic are $\mathbb Z_{p^4}$ and $\mathbb Z_{p^2q}$. By Theorem~\ref{intersecting graph t233}(1) and using \eqref{e4}-\eqref{33}, Figure~\ref{int f1}(b),~\ref{int f1}(c),~\ref{int f1}(d), Figure~\ref{fig:int f11}, there is no group $G$ such that $\mathscr{I}(G)$  is unicyclic. Thus, the proof of $(1)$ follows.

\noindent Proof of parts $(2)$ to $(14)$ of this Corollary are similar to the proof of part $(1)$.

If $\mathscr{I}(G)$ contains $C_3$, then obviously $girth(G)=3$. Now assume that $\mathscr{I}(G)$ is $C_3$-free.  If $G$ is cyclic, then by Theorem~\ref{l1}(4), $\mathscr{I}(G)$ is acyclic, since $\mathscr{I}(Z_{p^\alpha})\cong K_{p^{\alpha -1}}(\alpha = 2,3)$ and $\mathscr{I}(Z_{pq}) \cong \overline{K}_2$.  If $G$ is non-cyclic, then by Theorem~\ref{intersecting graph t233}(2),  $\mathscr{I}(G)$ is also acyclic. So  in both these cases $girth(G)=\infty$. This completes the proof.
\end{proof}

 \begin{rem}\normalfont

 \begin{enumerate}[\normalfont (i)]

 \item In Theorem~\ref{intersecting graph t12}, we show that all the projective-planar intersection graphs of subgroups of groups are toroidal, which is not the case for arbitrary graphs (e.g., see pp.\,367-368 and Figure~13.33 in \cite{KK2005}).

 \item In \cite{sel}, Selc¸uk Kayacan \emph{et al.} showed that the intersection graph of subgroups of groups other than those listed in Theorem~\ref{1000}, the group $G_2$ given in Case 1b in the proof of Proposition~\ref{intersecting graph t4}, and the group of order $p^2qr$ given in part (2) of Proposition~\ref{intersecting graph t8} contains $K_5$ as a subgraph. Also they showed that the intersection graph of subgroups of these two groups contains $K_{3,3}$ as a subgraph. In this paper, we proved by using another method that the intersection graph of subgroups of the first group contains $K_5$ as a subgraph, but not the second group. Thus it follows that the class of all groups having $K_5$-free intersection graph of subgroups properly contains the class of all groups having planar intersection graph of subgroups.
 \item In \cite{akbari_2} Akbari \emph{et al.} classified all groups whose intersection graphs of subgroups are one of $C_3$-free, tree, totally disconnected. Also they obtained the girth of the intersection graph of groups. In this paper, we proved these results for finite groups in a different method.
 \end{enumerate}
 \end{rem}

 In the next result we characterize some finite groups by using their intersection graph of subgroups.

\begin{cor}\label{intersecting graph c2}
Let $G$ be a group and $p, q, r$ are distinct primes. Then
\begin{enumerate}[\normalfont (1)]
\item $\mathscr{I}(G)\cong \mathscr{I}(M_8)$ if and only if $G\cong M_8$.
\item $\mathscr{I}(G)\cong \mathscr{I}(\mathbb Z_q\rtimes_2 \mathbb Z_{p^2})$ if and only if $G\cong \mathbb Z_q\rtimes_2 \mathbb Z_{p^2}$;
\item  $\mathscr{I}(G)\cong \mathscr{I}(\mathcal{G}_1)$ if and only if $G\cong \mathcal{G}_1$, where $\mathcal {G}_1$ is the group described in Theorem~\ref{1000}.
\item $\mathscr{I}(G)\cong \mathscr{I}(A_4)$ if and only if $G\cong A_4$;
\item $\mathscr{I}(G)\cong \mathscr{I}(\mathcal{G}_3)$ if and only if $G\cong \mathcal{G}_2$, where  $\mathcal {G}_2$ is the group described in Theorem~\ref{1000};
\item $\mathscr{I}(G)\cong \mathscr{I}(\mathcal{G}_2)$ if and only if $G\cong \mathcal{G}_3$, where  $\mathcal {G}_3$ is the group described in Theorem~\ref{1000};
\item $\mathscr{I}(G)\cong \mathscr{I}(G_1)$ if and only if $G\cong G_1$, where $G_1$ is the group of order $p^2qr$ described in Theorem~\ref{intersecting graph t233}(1).
\end{enumerate}
\end{cor}
\begin{proof} (1)-(7): By Theorem~\ref{intersecting graph t233}(1), we can see that the groups mentioned in this Corollary are having $K_5$-free intersection graph of subgroups. Also by
\eqref{e4}-\eqref{33}, Figure~\ref{int f1}(b)-(d), Figure~\ref{fig:int f11}, the intersection graphs of subgroups of these groups are unique. Since an infinite group has infinite number of subgroups, so its intersection graph of subgroups can not be isomorphic to any of the intersection graphs of subgroups of these groups.
If the intersection graph of a given finite group is isomorphic to any of these intersection graphs of subgroups of these groups, then by the uniqueness of these graphs, the given group must be isomorphic to that corresponding group. Proof of the result follows from this fact.
\end{proof}

\begin{thm}\label{intersection graph 300}
Let $G$ be a group. Then $\theta(\mathscr{I}(G))=m$, where $m$ is the number of prime order subgroups of groups of $G$.
\end{thm}
\begin{proof} Let $G$ be  finite and $|G|=p_1^{\alpha_1}p_2^{\alpha_2}\ldots p_k^{\alpha_k}$, where $p_i$'s are distinct primes, and $\alpha_i\geq 1$.
For each $i=1$, 2, $\ldots$, $k$, let $t_i$ be the number of subgroups of order $p_i$ and let $H(j,p_i)$, $j=1$, 2, $\ldots$, $t_i$ be a
subgroup of $G$
of order $p_i$. For each $j=1$, 2, $\ldots$, $t_i$, let $\mathcal S(j,p_i)$ be the set of all  subgroups of $G$ having $H(j,p_i)$ in common.
Clearly $\mathcal S(j, p_i)$ forms a clique in $\mathscr{I}(G)$ and $\mathcal S:=\{\mathcal S(j,p_i)~|~i=1$, 2, $\ldots$, $k$ and $j=1$, 2, $\ldots$, $t_i\}$
forms a clique cover of $\mathscr{I}(G)$ with  $|\mathcal S|=t_1+t_2+\ldots+t_k$. Therefore, $\Theta(\mathscr{I}(G))\leq |\mathcal S|$. Let $\mathcal{T}$ be a clique cover of $\mathscr{I}(G)$ such that $\Theta(\mathscr{I}(G))=|\mathcal{T}|$. If $|\mathcal{T}|<|\mathcal S|$, then by pigeonhole principle, $\mathcal{T}$ has a clique which contains atleast two subgroups, say $H(j,p_i)$, $H(l,p_r)$, for some $i\neq r$,
$j\in \{1$, 2, $\ldots$, $t_i\}$, $l\in \{1$, 2, $\ldots$, $t_r\}$, which is not possible, since $H(j,p_i)$ and $H(l,p_r)$ are not adjacent in $\mathscr{I}(G)$. So $|\mathcal{T}|=|\mathcal S|=$ the number of prime order subgroups of $G$. A similar argument also works when $G$ is infinite.
\end{proof}

Zelinka shown in \cite{zelinka} that for any group $G$, $\alpha(\mathscr(G))=m$, where $m$ is the number of prime order proper subgroups of $G$. This fact together with Theorem~\ref{intersection graph 300} gives the following result.
\begin{cor}\label{124}
If $G$ is a group, then $\mathscr{I}(G)$ is weakly $\alpha$-perfect.
\end{cor}
\section*{Acknowledgement}
The authors would like to thank Professor Ronald Solomon, Department of Mathematics, Ohio State University,  Columbus,  United States for his valuable suggestions to complete this work.

\end{document}